\documentclass{amsart}
\usepackage{latexsym,amssymb,amsmath,amsthm,amscd,graphicx}
\usepackage{esint} 
\usepackage{color}
\usepackage{braket}
\numberwithin{equation}{section}
\newtheorem{theorem}{Theorem}[section]
\newtheorem{lemma}[theorem]{Lemma}
\newtheorem{corollary}[theorem]{Corollary}
\newtheorem{proposition}[theorem]{Proposition}

\theoremstyle{definition}

\newtheorem{remark}[theorem]{Remark}

\theoremstyle{remark}

\usepackage{comment}

\begin{document}

\title[Atomic decomposition for Morrey-Lorentz spaces]{Atomic decomposition for Morrey-Lorentz spaces}
\author{Naoya Hatano}
\address[Naoya Hatano]{Department of Mathematics, Chuo University, 1-13-27, Kasuga, Bunkyo-ku, Tokyo 112-8551, Japan}
\email[Naoya Hatano]{n.hatano.chuo@gmail.com}
\maketitle

\begin{abstract}
In this paper, we consider the atomic decomposition for Morrey-Lorentz spaces and applications.
Morrey-Lorentz spaces, which have structures of Morrey spaces,  Lorentz spaces and their weak-type spaces, are introduced by M. A. Ragusa in 2012.
Our study gave some extension of the atomic decomposition to Morrey-Lorentz spaces.
As an application, the Olsen inequality can be obtained more sharpness.
\end{abstract}



\section{Introduction}

The purpose of this paper is to give the atomic decomposition for Morrey-Lorentz spaces.
Morrey-Lorentz spaces, which are an extension of Morrey and Lorentz spaces, were introduced by Ragusa \cite{Ragusa12} in 2012.
To introduce the Morrey-Lorentz spaces, we recall the definition of the Lorentz spaces.
In 1950, Lorentz introduced the Lorentz spaces in \cite{Lorentz50}.
For $\alpha>0$ and $t>0$, we define the distribution function $\lambda_f$ and the rearrangement function $f^\ast$ by
\begin{equation*}
\lambda_f(\alpha)
\equiv
|\{x\in{\mathbb R}^n:|f(x)|>\alpha\}|,
\quad
f^\ast(t)
\equiv
\inf\{\alpha>0:\lambda_f(\alpha)\le t\},
\end{equation*}
respectively.
Then for $0<p,q\le\infty$, the {\it Lorentz quasi-norm} $\|\cdot\|_{L^{p,q}}$ is defined by
\begin{equation*}
\|f\|_{L^{p,q}}\equiv
\begin{cases}
\displaystyle\left(\int_0^\infty(t^{\frac1p}f^\ast(t))^q\,\frac{{\rm d}t}t\right)^{\frac1q},
& \text{$0<p,q<\infty$},\\
\displaystyle\sup_{t>0}t^{\frac1p}f^\ast(t),
& \text{$0<p\le\infty$, $q=\infty$},
\end{cases}
\end{equation*}
and the {\it Lorentz space} $L^{p,q}({\mathbb R}^n)$ is the set 
of all measurable functions $f$ 
with the finite quasi-norm $\|\cdot\|_{L^{p,q}}$.
We remark that the Lorentz spaces $L^{p,p}({\mathbb R}^n)$ and $L^{p,\infty}({\mathbb R}^n)$ coincide with the Lebesgue space $L^p({\mathbb R}^n)$ and the weak Lebesgue space ${\rm W}L^p({\mathbb R}^n)$, respectively.

In addition, we will recall the definition of Morrey-Lorentz spaces.
Let $0<q\le p<\infty$ and $0<r\le\infty$.
We define the {\it Morrey-Lorentz space} ${\mathcal M}^p_{q,r}({\mathbb R}^n)$ by
\begin{equation*}
{\mathcal M}^p_{q,r}({\mathbb R}^n)
\equiv
\left\{
f\in L^0({\mathbb R}^n)
:
\|f\|_{{\mathcal M}^p_{q,r}}
\equiv
\sup_{Q\in{\mathcal Q}({\mathbb R}^n)}|Q|^{\frac1p-\frac1q}\|f\chi_Q\|_{L^{q,r}}
<\infty
\right\},
\end{equation*}
where the class $L^0({\mathbb R}^n)$ is the set of all measurable functions defined on ${\mathbb R}^n$ and ${\mathcal Q}({\mathbb R}^n)$ denotes the family of all cubes with parallel to coordinate axis in ${\mathbb R}^n$.
We note that the Morrey-Lorentz spaces ${\mathcal M}^p_{p,r}({\mathbb R}^n)$ and ${\mathcal M}^p_{q,q}({\mathbb R}^n)$ consist with the Lorentz space $L^{p,q}({\mathbb R}^n)$ and the Morrey space ${\mathcal M}^p_q({\mathbb R}^n)$, respectively.
Especially, ${\mathcal M}^p_{q,\infty}({\mathbb R}^n)$ corresponds to the weak Morrey space ${\rm W}{\mathcal M}^p_q({\mathbb R}^n)$ which the weak Morrey quasi-norm $\|\cdot\|_{{\rm W}{\mathcal M}^p_q}$ is defined by
\begin{equation*}
\|f\|_{{\rm W}{\mathcal M}^p_q}
\equiv
\sup_{\lambda>0}\lambda\|\chi_{\{x\in{\mathbb R}^n:|f(x)|>\lambda\}}\|_{{\mathcal M}^p_q}.
\end{equation*}

The set of all dyadic cubes
is denoted by ${\mathcal D}({\mathbb R}^n)${\rm;}
\begin{equation*}
{\mathcal D}_m({\mathbb R}^n)
\equiv
\left\{
Q_{m k}\equiv\prod_{i=1}^n\left[\frac{k_i}{2^m},\frac{k_i+1}{2^m}\right)
:
\text{$k=(k_1,\ldots,k_n)\in{\mathbb Z}^n$}
\right\},
\quad
{\mathcal D}({\mathbb R}^n)
\equiv
\bigcup_{j\in{\mathbb Z}}{\mathcal D}_j({\mathbb R}^n)
\end{equation*}
for $m\in{\mathbb Z}$.
Then by the simple consideration, we have that
\begin{equation*}
\|f\|_{{\mathcal M}^p_{q,r}}
\sim
\sup_{Q\in{\mathcal D}({\mathbb R}^n)}|Q|^{\frac1p-\frac1q}
\|f\chi_Q\|_{L^{q,r}},
\end{equation*}
for any measurable function $f$ in ${\mathcal M}^p_{q,r}({\mathbb R}^n)$.
For $E({\mathbb R}^n)=L^{q,r}({\mathbb R}^n)$ or ${\rm W}L^q({\mathbb R}^n)$ with the quasi-norm $\|\cdot\|_E$,
\begin{equation*}
E_{\rm loc}({\mathbb R}^n)
\equiv
\{
f\in L^0({\mathbb R}^n)
:
\text{$\|f\chi_K\|_E<\infty$, for all compact sets $K$ in ${\mathbb R}^n$}
\}.
\end{equation*}

We aim here to prove the following result which is extension of the atomic decomposition to Morrey-Lorentz spaces.

\begin{theorem}\label{thm:200308-1}
Suppose that the parameters $p,q,r,s,t,v$ satisfy
\begin{equation*}
0<q \le p<\infty, \quad
0<r\le\infty, \quad
0<t \le s<\infty, \quad
0<v\le1,
\end{equation*}
\begin{equation*}
q<t, \quad p<s, \quad v<\min(q,r).
\end{equation*}
Assume that
$\{Q_j\}_{j=1}^{\infty} \subset {\mathcal Q}({\mathbb R}^n)$,
$\{a_j\}_{j=1}^{\infty} \subset {\rm W}{\mathcal M}^s_t({\mathbb R}^n)$,
and
$\{\lambda_j\}_{j=1}^{\infty} \subset[0,\infty)$
fulfill
\begin{equation*}
\|a_j\|_{{\rm W}{\mathcal M}^s_t}\le|Q_j|^{\frac1s}, \quad {\rm supp}(a_j) \subset Q_j, \quad
\left\|\left(\sum_{j=1}^{\infty}(\lambda_j\chi_{Q_j})^v\right)^{\frac1v}\right\|_{{\mathcal M}^p_{q,r}}
<\infty.
\end{equation*}
Then, $f=\sum_{j=1}^{\infty} \lambda_j a_j$ converges a.e. and satisfies
\begin{equation}\label{eq:200308-1.3}
\|f\|_{{\mathcal M}^p_{q,r}}
\lesssim_{p,q,r,s,t,v}
\left\|\left(\sum_{j=1}^{\infty}(\lambda_j\chi_{Q_j})^v\right)^{\frac1v}\right\|_{{\mathcal M}^p_{q,r}}.
\end{equation}
In particular, $f=\sum_{j=1}^{\infty} \lambda_j a_j$ converges in $L_{\rm loc}^{q,r}({\mathbb R}^n)$ if $r<\infty$ and in $L^{p,r}({\mathbb R}^n)$ if $p=q$ and $r<\infty$.
\end{theorem}

In this theorem, we can take the atoms $\{a_j\}_{j=1}^\infty$ from the large space, the weak Morrey space ${\rm W}{\mathcal M}^s_t({\mathbb R}^n)$.
The difference between Morrey spaces and weak Morrey spaces in \cite{GHNS18}.
In addition, we can choose the parameter $v$ freely.\\

It is possible to rewrite Theorem \ref{thm:200308-1} to Lorentz spaces and weak Morrey spaces, respectively, as follows:

\begin{corollary}
Suppose that the parameters $p,r,s,t,v$ satisfy
\begin{equation*}
0<p<t\le s<\infty, \quad
0<r\le\infty, \quad
0<v\le1, \quad
v<\min(p,r).
\end{equation*}
Assume that
$\{Q_j\}_{j=1}^{\infty} \subset {\mathcal Q}({\mathbb R}^n)$,
$\{a_j\}_{j=1}^{\infty} \subset {\rm W}{\mathcal M}^s_t({\mathbb R}^n)$,
and
$\{\lambda_j\}_{j=1}^{\infty} \subset[0,\infty)$
fulfill
\begin{equation*}
\|a_j\|_{{\rm W}{\mathcal M}^s_t}\le|Q_j|^{\frac1s}, \quad {\rm supp}(a_j) \subset Q_j, \quad
\left\|\left(\sum_{j=1}^{\infty} (\lambda_j \chi_{Q_j})^v\right)^{\frac1v}\right\|_{L^{p,r}}
<\infty.
\end{equation*}
If $v<\min(p,r)$, then
$f=\sum_{j=1}^{\infty} \lambda_j a_j$
converges a.e.
and satisfies
\begin{equation*}
\|f\|_{L^{p,r}}
\lesssim_{p,r,s,t,v}
\left\|\left(\sum_{j=1}^{\infty}(\lambda_j\chi_{Q_j})^v\right)^{\frac1v}\right\|_{L^{p,r}}.
\end{equation*}
In particular, $f=\sum_{j=1}^{\infty} \lambda_j a_j$ converges in $L^{p,r}({\mathbb R}^n)$ if $r<\infty$.
\end{corollary}

\begin{corollary}
Suppose that the parameters $p,q,s,t,v$ satisfy
\begin{equation*}
0<q\le p<\infty, \quad
0<t\le s<\infty, \quad
0<v\le1, \quad
p<s, \quad v<q<t.
\end{equation*}
Assume that
$\{Q_j\}_{j=1}^{\infty} \subset {\mathcal Q}({\mathbb R}^n)$,
$\{a_j\}_{j=1}^{\infty} \subset {\rm W}{\mathcal M}^s_t({\mathbb R}^n)$
and
$\{\lambda_j\}_{j=1}^{\infty} \subset[0,\infty)$
fulfill
\begin{equation*}
\|a_j\|_{{\rm W}{\mathcal M}^s_t}\le|Q_j|^{\frac1s}, \quad {\rm supp}(a_j) \subset Q_j, \quad
\left\|
\left(\sum_{j=1}^{\infty} (\lambda_j \chi_{Q_j})^v\right)^{\frac1v}
\right\|_{{\rm W}{\mathcal M}^p_q}
<\infty.
\end{equation*}
Then
$f=\sum_{j=1}^{\infty} \lambda_j a_j$
converges a.e.
and satisfies
\begin{equation*}
\|f\|_{{\rm W}{\mathcal M}^p_q}
\lesssim_{p,q,s,t,v}
\left\|
\left(\sum_{j=1}^{\infty}(\lambda_j\chi_{Q_j})^v\right)^{\frac1v}
\right\|_{{\rm W}{\mathcal M}^p_q}.
\end{equation*}
\end{corollary}

The next assertion concerns the decomposition of functions in ${\mathcal M}^p_{q,r}({\mathbb R}^n)$.
Hereafter, we write ${\mathbb N}_0\equiv{\mathbb N} \cup \{0\}$.
Denote by ${\mathcal P}_K({\mathbb R}^n)$ the set of all polynomial functions with degree less than or equal to $K$.
The set ${\mathcal P}_K({\mathbb R}^n)^\perp$ denotes the set of all $f\in L^0({\mathbb R}^n)$ for which $\langle \cdot \rangle^Kf\in L^1({\mathbb R}^n)$ and $\displaystyle\int_{{\mathbb R}^n}x^\alpha f(x){\rm d}x=0$ for any $\alpha \in {\mathbb N}_0^n$ with $|\alpha| \le K$, where $\langle \cdot \rangle=(1+|\cdot|^2)^{\frac{1}{2}}$.

\begin{theorem}\label{thm:200308-2}
Let $1<q\le p<\infty$, $0<r\le\infty$, $K\in{\mathbb N}_0$, and $f\in{\mathcal M}^p_{q,r}({\mathbb R}^n)$.
Then, there exists a triplet
$\{\lambda_j\}_{j=1}^\infty\subset[0,\infty)$,
$\{Q_j\}_{j=1}^\infty\subset{\mathcal Q}({\mathbb R}^n)$,
and
$
\{a_j\}_{j=1}^\infty\subset L^\infty({\mathbb R}^n)\cap{\mathcal P}_K^\perp({\mathbb R}^n)$
such that
$f=\sum_{j=1}^\infty\lambda_ja_j$ in ${\mathcal S}'({\mathbb R}^n)$
and that, for all $v>0$,
\begin{equation}\label{eq:200308-2.1}
|a_j|\le\chi_{Q_j}, \quad
\left\|
\left(\sum_{j=1}^\infty(\lambda_j\chi_{Q_j})^v\right)^{\frac1v}
\right\|_{{\mathcal M}^p_{q,r}}
\lesssim_v
\|f\|_{{\mathcal M}^p_{q,r}}.
\end{equation}
\end{theorem}

We can also Theorem \ref{thm:200308-2} to Lorentz spaces and weak Morrey spaces, respectively, as follows:

\begin{corollary}
Let $1<p<\infty$, $0<r\le\infty$, $K\in{\mathbb N}_0$, let $f\in L^{p,r}({\mathbb R}^n)$.
Then there exists a triplet
$\{\lambda_j\}_{j=1}^\infty\subset[0,\infty)$,
$\{Q_j\}_{j=1}^\infty\subset{\mathcal Q}({\mathbb R}^n)$
and
$
\{a_j\}_{j=1}^\infty\subset L^\infty({\mathbb R}^n)\cap{\mathcal P}_K^\perp({\mathbb R}^n)$
such that
$f=\sum_{j=1}^\infty\lambda_ja_j$ in ${\mathcal S}'({\mathbb R}^n)$
and that, for all $v>0$,
\begin{equation*}
|a_j|\le\chi_{Q_j}, \quad
\left\|
\left(\sum_{j=1}^\infty(\lambda_j\chi_{Q_j})^v\right)^{\frac1v}
\right\|_{L^{p,r}}
\lesssim_v
\|f\|_{L^{p,r}}.
\end{equation*}
\end{corollary}

\begin{corollary}
Let $1<q\le p<\infty$, $K\in{\mathbb N}_0$, let $f\in {\rm W}{\mathcal M}^p_q({\mathbb R}^n)$.
Then there exists a triplet
$\{\lambda_j\}_{j=1}^\infty\subset[0,\infty)$,
$\{Q_j\}_{j=1}^\infty\subset{\mathcal Q}({\mathbb R}^n)$
and
$
\{a_j\}_{j=1}^\infty\subset L^\infty({\mathbb R}^n)\cap{\mathcal P}_K^\perp({\mathbb R}^n)$
such that
$f=\sum_{j=1}^\infty\lambda_ja_j$ in ${\mathcal S}'({\mathbb R}^n)$
and that, for all $v>0$,
\begin{equation*}
|a_j|\le\chi_{Q_j}, \quad
\left\|
\left(\sum_{j=1}^\infty(\lambda_j\chi_{Q_j})^v\right)^{\frac1v}
\right\|_{{\rm W}{\mathcal M}^p_q}
\lesssim_v
\|f\|_{{\rm W}{\mathcal M}^p_q}.
\end{equation*}
\end{corollary}

Theorems \ref{thm:200308-1} and \ref{thm:200308-2} are special cases of Theorems \ref{thm:200308-3} and \ref{thm:200308-4} to follow, respectively, which concerns the decomposition of Hardy-Morrey-Lorentz spaces.
Moreover, Theorem \ref{thm:200929-1} stands for the critical case $t=1$ in Theorem \ref{thm:200308-3}.
Recall that, for $0<q\le p<\infty$ and $0<r\le\infty$, the {\it Hardy-Morrey-Lorentz space} $H{\mathcal M}^p_{q,r}({\mathbb R}^n)$ is defined to be the set of all $f\in{\mathcal S}'({\mathbb R}^n)$ for which the quasi-norm
$
\|f\|_{H{\mathcal M}^p_{q,r}}\equiv\|\sup_{t>0}|e^{t\Delta}f|\|_{\mathcal{M}^p_{q,r}}
$
is finite, where $e^{t\Delta}f$ stands for the heat expansion of $f$ for $t>0$;
\begin{equation*}
e^{t\Delta}f(x)
=
\left\langle\frac1{\sqrt{(4\pi t)^n}}\exp\left(-\frac{|x-\cdot|^2}{4t}\right),f\right\rangle,
\quad x\in{\mathbb R}^n.
\end{equation*}
In addition, $H{\mathcal M}^p_{q,\infty}({\mathbb R}^n)$ coincides with the Hardy-weak Morrey space $H{\rm W}{\mathcal M}^p_q({\mathbb R}^n)$, introduced by Ho in \cite{Ho17}.

\begin{theorem}\label{thm:200308-3}
Suppose that the parameters $p,q,r,s,t,v$ satisfy
\begin{equation*}
0<q \le p<\infty, \quad 0<r\le\infty, \quad
1<t \le s<\infty, \quad
0<v\le1,
\end{equation*}
\begin{equation*}
q<t, \quad p<s, \quad v<\min(q,r),
\end{equation*}
Write $d_v\equiv[n(1/v-1)]$.
Assume that
$\{Q_j\}_{j=1}^{\infty} \subset {\mathcal Q}({\mathbb R}^n)$,
$\{a_j\}_{j=1}^{\infty} \subset {\rm W}{\mathcal M}^s_t({\mathbb R}^n)\cap{\mathcal P}_{d_v}({\mathbb R}^n)^\perp$
and
$\{\lambda_j\}_{j=1}^{\infty} \subset[0,\infty)$
fulfill
\begin{equation*}
\|a_j\|_{{\rm W}{\mathcal M}^s_t}\le|Q_j|^{\frac1s}, \quad
{\rm supp}(a_j) \subset Q_j, \quad
\left\|
\left(\sum_{j=1}^{\infty}(\lambda_j\chi_{Q_j})^v\right)^{\frac1v}
\right\|_{{\mathcal M}^p_{q,r}}
<\infty.
\end{equation*}
Then
$f=\sum_{j=1}^{\infty} \lambda_j a_j$
converges in
${\mathcal S}'({\mathbb R}^n)$
and satisfies
\begin{equation}\label{eq:200308-3.3}
\|f\|_{H{\mathcal M}^p_{q,r}}
\lesssim_{p,q,r,s,t}
\left\|
\left(\sum_{j=1}^{\infty}(\lambda_j\chi_{Q_j})^v\right)^{\frac1v}
\right\|_{{\mathcal M}^p_{q,r}}.
\end{equation}
\end{theorem}

\begin{theorem}\label{thm:200929-1}
Suppose that the parameters $p,q,r,s,v$ satisfy
\begin{equation*}
0<q \le p<\infty, \quad 0<r\le\infty, \quad
1\le s<\infty, \quad
0<v\le1, \quad
q<1, \quad p<s, \quad v<\min(q,r).
\end{equation*}
Write and $d_v\equiv[n(1/v-1)]$.
Assume that
$\{Q_j\}_{j=1}^{\infty} \subset {\mathcal Q}({\mathbb R}^n)$,
$\{a_j\}_{j=1}^{\infty} \subset{\mathcal M}^s_1({\mathbb R}^n)\cap{\mathcal P}_{d_v}({\mathbb R}^n)^\perp$
and
$\{\lambda_j\}_{j=1}^{\infty} \subset[0,\infty)$
fulfill
\begin{equation*}
\|a_j\|_{{\mathcal M}^s_1}\le|Q_j|^{\frac1s}, \quad
{\rm supp}(a_j) \subset Q_j, \quad
\left\|
\left(\sum_{j=1}^{\infty}(\lambda_j\chi_{Q_j})^v\right)^{\frac1v}
\right\|_{{\mathcal M}^p_{q,r}}
<\infty.
\end{equation*}
Then
$f=\sum_{j=1}^{\infty} \lambda_j a_j$
converges in
${\mathcal S}'({\mathbb R}^n)$
and satisfies
\begin{equation}\label{eq:200929-1.3}
\|f\|_{H{\mathcal M}^p_{q,r}}
\lesssim_{p,q,r,s}
\left\|
\left(\sum_{j=1}^{\infty}(\lambda_j\chi_{Q_j})^v\right)^{\frac1v}
\right\|_{{\mathcal M}^p_{q,r}}.
\end{equation}
\end{theorem}

\begin{theorem}\label{thm:200308-4}
Suppose that the real parameters $p$, $q$, $r$ and $K$ satisfy
\begin{equation*}
0<q\le p<\infty, \quad
0<r\le\infty, \quad
K\in{\mathbb N}_0\cap\left(\frac n{q_0}-n-1,\infty\right),
\end{equation*}
where $q_0\equiv\min(1,q)$.
Let $f\in H{\mathcal M}^p_{q,r}({\mathbb R}^n)$.
Then there exists a triplet
$\{\lambda_j\}_{j=1}^\infty\subset[0,\infty)$,
$\{Q_j\}_{j=1}^\infty\subset{\mathcal Q}({\mathbb R}^n)$
and
$
\{a_j\}_{j=1}^\infty\subset L^\infty({\mathbb R}^n)\cap{\mathcal P}_K^\perp({\mathbb R}^n)
$
such that
$f=\sum_{j=1}^\infty\lambda_ja_j$ in ${\mathcal S}'({\mathbb R}^n)$
and that, for all $v>0$,
\begin{equation*}
|a_j|\le\chi_{Q_j}, \quad
\left\|
\left(\sum_{j=1}^\infty(\lambda_j\chi_{Q_j})^v\right)^{\frac1v}
\right\|_{{\mathcal M}^p_{q,r}}
\lesssim_v
\|f\|_{H{\mathcal M}^p_{q,r}}.
\end{equation*}
\end{theorem}

Concerning ${\mathcal M}^p_{q,r}({\mathbb R}^n)$ and $H{\mathcal M}^p_{q,r}({\mathbb R}^n)$, we have the following assertion:

\begin{proposition}\label{prop:200308-1}
Let $1\le q\le p<\infty$ and $0<r\le\infty$.
\begin{enumerate}
\item If $f \in {\mathcal M}^p_{q,r}({\mathbb R}^n)$ and $q>1$, then $f \in H{\mathcal M}^p_{q,r}({\mathbb R}^n)$.
\item Assume that $q>1$ or $q=1\ge r$.
If $f \in H{\mathcal M}^p_{q,r}({\mathbb R}^n)$, then $f$ can be represented by a locally integrable function and the representative belongs to ${\mathcal M}^p_{q,r}({\mathbb R}^n)$.
\end{enumerate}
\end{proposition}

\begin{remark}\label{rem:210427-1}
\begin{enumerate}
\item It is noteworthy that the case $q=1\ge r$ in this proposition covers a result of \cite{HMS20} as a special case of $r=1$.
\item It is not clear that $H{\mathcal M}^p_{1,r}({\mathbb R}^n)\hookrightarrow{\mathcal M}^p_{1,r}({\mathbb R}^n)$ for $r>1$.
Because the embedding $H{\mathcal M}^p_{1,r}({\mathbb R}^n)\hookrightarrow L_{\rm loc}^1({\mathbb R}^n)$ fails.
Let $n=1$.
Here, we explain that $L^{1,r}({\mathbb R})$ does not embed into $L_{\rm loc}^1({\mathbb R})$.
In fact, if we let
\begin{equation*}
f
=
\sum_{j=2}^\infty\frac{2^j}j\chi_{[2^{-j-1},2^{-j})},
\end{equation*}
then the necessary and sufficiently condition of $f\in L^{1,r}({\mathbb R})$ is $r>1$ (see \cite[p. 56]{BSNS17}).
Meanwhile, for any $k\in{\mathbb N}$, we see that
\begin{equation*}
\|f\chi_{[0,2^{-k})}\|_{L^1}=\infty.
\end{equation*}
\end{enumerate}
\end{remark}

Hardy-Morrey-Lorentz spaces admit a characterization
by using the grand maximal operator.
To formulate the result,
we recall the following two fundamental notions.
\begin{enumerate}
\item Topologize ${\mathcal S}({\mathbb R}^n)$ by norms $\{p_N\}_{N \in {\mathbb N}}$
	given by
\begin{equation*}
p_N(\varphi)
\equiv
\sum_{|\alpha|\le N}\sup_{x\in{\mathbb R}^n}(1+|x|)^N|\partial^{\alpha}\varphi(x)|
\end{equation*}
for each $N \in {\mathbb N}$.
Define ${\mathcal F}_N\equiv\{\varphi\in{\mathcal S}({\mathbb R}^n):p_N(\varphi)\le 1\}$.
\item Let $f \in {\mathcal S}'({\mathbb R}^n)$.
The grand maximal operator ${\mathcal M}f$
is given by
\begin{equation*}
{\mathcal M}f(x)
\equiv
\sup
\{|t^{-n}\psi(t^{-1}\cdot)\ast f(x)|
:
\text{$t>0$, $\psi\in{\mathcal F}_N$}\},
\quad x \in {\mathbb R}^n,
\end{equation*}
where we choose and fix a large integer $N$.
\end{enumerate}

The Hardy-Morrey-Lorentz quasi-norm $\|\cdot\|_{H{\mathcal M}^p_{q,r}}$ is rewritten as follows.

\begin{proposition}\label{prop:200308-2}
Let $0<q\le p<\infty$ and $0<r\le\infty$.
Then
\begin{equation*}
\|{\mathcal M}f\|_{{\mathcal M}^p_{q,r}}
\sim
\|f\|_{H{\mathcal M}^p_{q,r}}
\end{equation*}
for all $f \in {\mathcal S}'({\mathbb R}^n)$.
\end{proposition}

In 2009, H. Jia and H. Wang \cite{JiWa09} provided the atomic decomposition for Morrey spaces.
T. Iida, Y. Sawano and H. Tanaka \cite{IST14} also gave the decomposition and its applications.
Later, many authors showed for various functional spaces; Lorentz spaces \cite{Parilov05-06,AbTo07,LYY16} Lebesgue spaces with variable exponent \cite{NaSa12,Sawano13,CrWa14} Orlicz-Morrey spaces of the third kind \cite{GHSN16}, local Morrey spaces \cite{BaSa14,GHS17}, generalized Morrey spaces \cite{AGNS17}, weak Morrey spaces \cite{Ho17}, mixed Morrey spaces \cite{NOSS20}, etc.

We organize the remaining part of the paper as follows:
In Section \ref{s2}, we introduce the necessary fundamental statements for Morrey-Lorentz spaces.
In Section \ref{s3}, we prove Propositions \ref{prop:200308-1} and \ref{prop:200308-1}.
In Section \ref{s4}, we prove Theorems \ref{thm:200308-1}, \ref{thm:200308-3} and \ref{thm:200929-1}.
In Section \ref{s5}, we prove Theorem \ref{thm:200308-4}.
In Section \ref{s6}, we give an application for our main results.

\section{Preliminaries}\label{s2}

\subsection{Fundamental properties for function spaces}

According to the paper \cite{Hunt66}, the H\"older inequality for Lorentz quasi-norms.

\begin{lemma}[{\cite[Theorem 4.5]{Hunt66}}]\label{lem:Lpq-Holder}
Assume that $0<p,p_1,p_2<\infty$ and $0<q,q_1,q_2$ $\le\infty$ satisfy
\begin{equation*}
\frac1p=\frac1{p_1}+\frac1{p_2}, \quad
\frac1q=\frac1{q_1}+\frac1{q_2}.
\end{equation*}
Then,
\begin{equation*}
\|f\cdot g\|_{L^{p,q}}
\lesssim
\|f\|_{L^{p_1,q_1}}\|g\|_{L^{p_2,q_2}}
\end{equation*}
for all $f\in L^{p_1,q_1}({\mathbb R}^n)$ and $g\in L^{p_2,q_2}({\mathbb R}^n)$.
\end{lemma}

The embedding relations for Morrey-Lorentz spaces are given by Ragusa in \cite{Ragusa12}.

\begin{proposition}[{\cite[Theorem 3.1]{Ragusa12}}]\label{prop:Mpqr-embedding}
The following assertions hold{\rm :}
\begin{itemize}
\item[{\rm (1)}] If $0<q\le p<\infty$ and $0<r_1\le r_2\le\infty$, then
\begin{equation*}
{\mathcal M}^p_{q,r_1}({\mathbb R}^n)
\hookrightarrow {\mathcal M}^p_{q,r_2}({\mathbb R}^n).
\end{equation*}
\item[{\rm (2)}] If $0<q_2<q_1\le p<\infty$ and $0<r_1,r_2\le\infty$, then
\begin{equation*}
{\mathcal M}^p_{q_1,r_1}({\mathbb R}^n)
\hookrightarrow {\mathcal M}^p_{q_2,r_2}({\mathbb R}^n).
\end{equation*}
\end{itemize}
\end{proposition}

Although we cannot calculate $\|\chi_E\|_{{\mathcal M}^p_{q,r}}$ for all measurable subsets $E$, we can do so for any cube $E$.

\begin{proposition}\label{prop:Mpqr-indicator}
Let $0<q\le p<\infty$ and $0<r\le\infty$, and let $Q\in{\mathcal Q}({\mathbb R}^n)$.
Then
\begin{equation*}
\|\chi_Q\|_{{\mathcal M}^p_{q,r}}
=
\left(\frac qr\right)^{\frac1r}|Q|^{\frac1p},
\end{equation*}
where we assume that $(q/r)^{1/r}=1$ for $r=\infty$.
\end{proposition}

\begin{proof}
As is mentioned in \cite[Example 1.4.8]{Grafakos14-1}, for each measurable set $E$,
\[
\|\chi_E\|_{L^{q,r}}
=
\left(\frac qr\right)^{\frac1r}|E|^{\frac1q}.
\]
It follows that
\begin{equation*}
\|\chi_Q\|_{{\mathcal M}^p_{q,r}}
=
\left(\frac qr\right)^{\frac1r}
\sup_{R\in{\mathcal Q}({\mathbb R}^n)}|R|^{\frac1p-\frac1q}|Q\cap R|^{\frac1q}.
\end{equation*}
Then, combining the estimates
\begin{equation*}
\sup_{R\in{\mathcal Q}({\mathbb R}^n)}|R|^{\frac1p-\frac1q}|Q\cap R|^{\frac1q}
\le
\sup_{R\in{\mathcal Q}({\mathbb R}^n)}|Q\cap R|^{\frac1p-\frac1q}|Q\cap R|^{\frac1q}
=
|Q|^{\frac1p}
\end{equation*}
and
\begin{equation*}
\sup_{R\in{\mathcal Q}({\mathbb R}^n)}|R|^{\frac1p-\frac1q}|Q\cap R|^{\frac1q}
\ge
\sup_{R\in{\mathcal Q}({\mathbb R}^n),\,R\subset Q}|R|^{\frac1p-\frac1q}|R|^{\frac1q}
=
|Q|^{\frac1p},
\end{equation*}
we obtain the desired result.
\end{proof}

Morrey-Lorentz quasi-norms satisfy the Fatou property.

\begin{lemma}[Fatou property for Morrey-Lorentz space]\label{lem:Fatou Mpqr}
Let $0<q\le p<\infty$ and $0<r\le\infty$, and let $\{f_j\}_{j=1}^\infty\subset L^0({\mathbb R}^n)$ be a nonnegative collection such that $f=\lim_{j\to\infty}f_j$ exists a.e.
Then, we have
\begin{equation*}
\|f\|_{{\mathcal M}^p_{q,r}}
\le
\liminf_{j\to\infty}\|f_j\|_{{\mathcal M}^p_{q,r}}.
\end{equation*}
\end{lemma}

\begin{proof}
For each $Q\in{\mathcal Q}({\mathbb R}^n)$,
\begin{equation*}
f_j\chi_Q\to f\chi_Q
\quad \text{a.e.},
\end{equation*}
and therefore, by the Fatou property for the Lorentz quasi-norm $\|\cdot\|_{L^{q,r}}$ (see \cite[Exercise 1.4.11]{Grafakos14-1}),
\begin{equation*}
\|f\chi_Q\|_{L^{q,r}}
\le
\liminf_{j\to\infty}\|f_j\chi_Q\|_{L^{q,r}}.
\end{equation*}
Consequently,
\begin{align*}
\|f\|_{{\mathcal M}^p_{q,r}}
&\le
\sup_{Q\in{\mathcal Q}({\mathbb R}^n)}
|Q|^{\frac1p-\frac1q}\liminf_{j\to\infty}\|f_j\chi_Q\|_{L^{q,r}}
\le
\liminf_{j\to\infty}
\sup_{Q\in{\mathcal Q}({\mathbb R}^n)}
|Q|^{\frac1p-\frac1q}\|f_j\chi_Q\|_{L^{q,r}}\\
&=
\liminf_{j\to\infty}\|f_j\|_{{\mathcal M}^p_{q,r}}.
\end{align*}
\end{proof}

\subsection{Maximal operators}

Let $0<\eta<\infty$ and $0<\theta\le\infty$.
For a measurable function $f$ defined on ${\mathbb R}^n$, define a function $M^{(\eta,\theta)}f$ by
\begin{equation*}
M^{(\eta,\theta)}f(x)
\equiv
\sup_{Q\in{\mathcal Q}({\mathbb R}^n)}\chi_Q(x)
\frac{\|f\chi_Q\|_{L^{\eta,\theta}}}{\|\chi_Q\|_{L^{\eta,\theta}}},
\quad x\in{\mathbb R}^n.
\end{equation*}
When $\theta=\eta$, we abbreviate $M^{(\eta,\eta)}=M^{(\eta)}$.
The operator $M^{(\eta,\theta)}$ is bounded on the Lorentz space $L^{p,q}({\mathbb R}^n)$.

\begin{proposition}{\cite[Proposition 2]{Hatano20-2}}\label{prop:190620-1}
Let $0<p,q\le\infty$, $0<\eta<\infty$ and $0<\theta\le\infty$.
If $\eta<p$, then the operator $M^{(\eta,\theta)}$ is a bounded operator from $L^{p,q}({\mathbb R}^n)$ to itself.
\end{proposition}

The critical case $\eta=p$ in this proposition, which the operator $M^{(p,q)}$ is bounded from $L^{p,q}({\mathbb R}^n)$ to ${\rm W}L^p({\mathbb R}^n)$ if and only if $0<q\le p<\infty$, is given in \cite[Theorem 6]{CHK82}.

The maximal operator $M^{(1,1)}$ is denoted by $M$, which is called the Hardy-Littlewood maximal operator.
It is well known that the boundedness of the operator $M$ holds on Morrey spaces.

\begin{proposition}[{\cite{ChFr87}}]\label{prop:200929-1}
Let $1\le q\le p<\infty$.
Then, the following assertions hold{\rm :}
\begin{itemize}
\item[{\rm (1)}] For all $f\in{\mathcal M}^p_1({\mathbb R}^n)$,
\begin{equation*}
\|Mf\|_{{\rm W}{\mathcal M}^p_1}
\lesssim
\|f\|_{{\mathcal M}^p_1}.
\end{equation*}
\item[{\rm (2)}] If $1<q\le p<\infty$, for all $f\in{\mathcal M}^p_q({\mathbb R}^n)$,
\begin{equation*}
\|Mf\|_{{\mathcal M}^p_q}\lesssim\|f\|_{{\mathcal M}^p_q}.
\end{equation*}
\end{itemize}
\end{proposition}

As an extension of Proposition \ref{prop:200929-1} (2), we now recall the boundedness of the operator $M$ on Morrey-Lorentz spaces.

\begin{proposition}[{\cite[Proposition 4]{Hatano20-2}}]\label{prop:190622-2}
Let $1<q\le p<\infty$ and $0<r\le\infty$.
Then we have
\begin{equation*}
\|Mf\|_{{\mathcal M}^p_{q,r}}\lesssim\|f\|_{{\mathcal M}^p_{q,r}}
\end{equation*}
for any measurable function $f$ defined on ${\mathbb R}^n$.
More generally, if $0<\eta<q\le p<\infty$, then
\begin{equation*}
\|M^{(\eta)}f\|_{{\mathcal M}^p_{q,r}}\lesssim\|f\|_{{\mathcal M}^p_{q,r}}
\end{equation*}
for any measurable function $f$ defined on ${\mathbb R}^n$.
\end{proposition}

We recall the Fefferman-Stein inequality for Morrey-Lorentz quasi-norm.
The case $r=\infty$ in Proposition \ref{prop:200324-1}, below, is not contained in \cite[Theorem 10 (1)]{Hatano20-2}, but, using Proposition 2.6 in \cite{Ho17}, we can check this case, similarly.

\begin{proposition}\label{prop:200324-1}
\begin{itemize}
\item[(1)] Let $1<q\le p<\infty$, $1<r\le \infty$ and $1<u<\infty$.
Then
\begin{equation*}
\left\|\left(
\sum_{j=1}^\infty (Mf_j)^u
\right)^{\frac1u}\right\|_{{\mathcal M}^p_{q,r}}
\lesssim
\left\|\left(
\sum_{j=1}^\infty |f_j|^u
\right)^{\frac1u}\right\|_{{\mathcal M}^p_{q,r}}
\end{equation*}
for all sequences of measurable functions $\{f_j\}_{j=1}^\infty$.
\item[(2)] Let $1<q\le p<\infty$ and $0<r\le\infty$.
Then
\begin{equation*}
\left\|\sup_{j\in{\mathbb N}}Mf_j\right\|_{{\mathcal M}^p_{q,r}}
\lesssim
\left\|\sup_{j\in{\mathbb N}}|f_j|\right\|_{{\mathcal M}^p_{q,r}}.
\end{equation*}
\end{itemize}
\end{proposition}

\subsection{Lorentz-block spaces}

It is known the predual space of Morrey-Lorentz space in \cite[Lemma 3.1]{Ferreira16} and \cite[Theorem 3.5]{Ho14}.
Let $1<q\le p<\infty$ and $1<r\le\infty$, and the conjugate number $p'$ of $p$ is defined by $\dfrac1p+\dfrac1{p'}=1$.
Then the predual space ${\mathcal H}^{p'}_{q',r'}({\mathbb R}^n)$ of the Morrey-Lorentz space ${\mathcal M}^p_{q,r}({\mathbb R}^n)$ is given by
\begin{equation*}
{\mathcal H}^{p'}_{q',r'}({\mathbb R}^n)
=
\left\{
g=\sum_{j=1}^\infty\mu_jb_j
\,:\,
\text{
$\{\mu_j\}_{j=1}^\infty\in\ell^1({\mathbb N})$,
each $b_j$ is a $(p',q',r')$-block
}
\right\}.
\end{equation*}
Here by \lq \lq a $(p',q',r')$-block" we mean an $L^{q',r'}({\mathbb R}^n)$-function supported on a cube $Q\in{\mathcal Q}({\mathbb R}^n)$ with $L^{q',r'}({\mathbb R}^n)$-norm less than or equal to $|Q|^{\frac1{q'}-\frac1{p'}}$.
The norm of ${\mathcal H}^{p'}_{q',r'}({\mathbb R}^n)$ is defined by
\begin{equation*}
\|g\|_{{\mathcal H}^{p'}_{q',r'}}
=
\inf\sum_{j=1}^\infty |\mu_j|,
\end{equation*}
where $\inf$ is over all admissible expressions above.
We remark that the norm equivalent
\begin{equation*}
\|f\|_{{\mathcal M}^p_{q,r}}
\sim
\sup
\left\{
\int_{{\mathbb R}^n}|f(x)g(x)|\,{\rm d}x
\,:\,
\|g\|_{{\mathcal H}^{p'}_{q',r'}}=1
\right\}
\end{equation*}
is obtained.
In particular, it is known that the predual space of the weak Morrey space ${\rm W}{\mathcal M}^p_q({\mathbb R}^n)={\mathcal M}^p_{q,\infty}({\mathbb R}^n)$ is ${\mathcal H}^{p'}_{q',1}({\mathbb R}^n)$ in \cite[Theorem 3.6]{Ho17} and \cite[Theorem 2.5]{SaEl-Sh18}.

\section{Proof of Propositions \ref{prop:200308-1} and \ref{prop:200308-2}}\label{s3}

\subsection{Proof of Proposition \ref{prop:200308-1}}

Let $w$ be a locally integrable function, and recall that $w$ is an $A_1$-weight when $Mw\lesssim w$.
We define the weighted $L^1$-space $L^1({\mathbb R}^n,w)$ by the space of all $f\in L^0({\mathbb R}^n)$ with finite norm
\begin{equation*}
\|f\|_{L^1(w)}
:=
\int_{{\mathbb R}^n}|f(x)|w(x)\,{\rm d}x.
\end{equation*}

To prove Proposition {\rm \ref{prop:200308-1}, we use the following lemmas and the atomic decomposition for the weighted Hardy space $H^1({\mathbb R}^n,w)$ given by Nakai and Sawano \cite{NaSa14}.

\begin{lemma}\label{lem:210910-1}
\begin{itemize}
\item[{\rm (1)}] For all $f\in{\mathcal S}({\mathbb R}^n)$,
\begin{equation}\label{eq:210919-1}
e^{t\Delta}f(x)\to f(x)
\quad \text{as} \quad
t\downarrow0.
\end{equation}
\item[{\rm (2)}] Let $w\in A_1$, and let $f\in L^1({\mathbb R}^n,w)$.
Then for each $x\in{\mathbb R}^n$ and $t>0$,
\begin{equation*}
\int_{{\mathbb R}^n}
\left|\frac1{\sqrt{(4\pi t)^n}}\exp\left(-\frac{|x-y|^2}{4t}\right)f(y)\right|
\,{\rm d}y
<\infty
\end{equation*}
holds.
Moreover, we have
\begin{equation}\label{eq:210913-2}
|f|\le\sup_{t>0}|e^{t\Delta}f|, \quad
\text{a.e.}
\end{equation}
\item[{\rm (3)}] For all $f\in{\mathcal S}'({\mathbb R}^n)$,
\[
e^{t\Delta}f\to f
\quad \text{in} \quad
{\mathcal S}'({\mathbb R}^n).
\]
\end{itemize}
\end{lemma}

\begin{proof}

(1) Because
\[
|\exp(-|z|^2)f(x-2\sqrt{t}z)|
\le
\exp(-|z|^2)\|f\|_{L^\infty}\in L^1({\mathbb R}^n),
\]
by the Lebesgue convergence theorem, we have
\begin{align*}
e^{t\Delta}f(x)
=
\frac1{\sqrt{\pi^n}}
\int_{{\mathbb R}^n}\exp(-|z|^2)f(x-2\sqrt{t}z)\,{\rm d}z
\to
\frac1{\sqrt{\pi^n}}
\int_{{\mathbb R}^n}\exp(-|z|^2)f(x)\,{\rm d}z
=
f(x)
\end{align*}
as $t\downarrow0$.
This proves \eqref{eq:210919-1}.\\

(2) To prove \eqref{eq:210913-2}, it suffices to show that the set
\begin{equation*}
E_k\equiv
\left\{
x\in{\mathbb R}^n
\,:\,
\underset{t\downarrow0}{\rm limsup}|e^{t\Delta}f(x)-f(x)|>\frac1k
\right\}
\end{equation*}
is a null set for all $k\in{\mathbb N}$ and $f\in L^1({\mathbb R}^n,w)$.
Fix $\varepsilon>0$, and take $g\in{\mathcal S}({\mathbb R}^n)$ such that
\[
\|f-g\|_{L^1(w)}<\varepsilon.
\]
Because the function
\[
(0,\infty)\ni\lambda
\longmapsto
\varphi(\lambda)\equiv\exp(-\lambda)\in(0,\infty)
\]
is positive and decreasing on $(0,\infty)$, we deduce from \cite[Proposition 2.7]{Duoandikoetxea01} that
\begin{equation*}
\sup_{t>0}|e^{t\Delta}f|
=
\sup_{t>0}
\left|\frac1{\sqrt{(4\pi t)^n}}\varphi\left(-\frac{|\cdot|^2}{4t}\right)\ast f\right|
\lesssim
\frac1{\sqrt{\pi^n}}\|\varphi(|\cdot|^2)\|_{L^1}Mf.
\end{equation*}
Thus, by \eqref{eq:210919-1}, we estimate
\begin{align*}
w(E_k)
&\le
w
\left(
\left\{\underset{t\downarrow0}{\rm limsup}|e^{t\Delta}[f-g]|>\frac1{2k}\right\}
\right)
+
w\left(\left\{|f-g|>\frac1{2k}\right\}\right)\\
&\lesssim
w\left(\left\{M[f-g]>\frac1{2k}\right\}\right)
+
w\left(\left\{|f-g|>\frac1{2k}\right\}\right).
\end{align*}
Applying the weak-type boundedness of $M$ on $L^1({\mathbb R}^n,w)$ (see, e.g., \cite[Theorem 7.1.9]{Grafakos14-1}) and Chebyshev's inequality, we conclude that
\begin{align*}
w(E_k)
\lesssim_{[w]_{A_1}}
4k\|f-g\|_{L^1(w)}
<
4k\varepsilon.
\end{align*}
We finish the proof of Lemma \ref{lem:210910-1} because $\varepsilon>0$ and $w(x)\,{\rm d}x$ and ${\rm d}x$ are mutually absolutely continuous.

(3) We omit the proof of this statement.
See \cite[Theorem 1.35]{Sawano18} for the discrete case.
A minor modification suffices for the continuous case.
The same argument applies to the Gaussian, although the Gaussian is not compactly supported.
\end{proof}

In addition, we use the following lemma:

\begin{lemma}\label{lem:211213-1}
Let $1\le p<\infty$.
Then the following assertions hold{\rm :}
\begin{itemize}
\item[{\rm (1)}] If $\delta\in(0,1)$ satisfies $\delta>1-1/p$,
\[
{\mathcal M}^p_1({\mathbb R}^n)\hookrightarrow L^1({\mathbb R}^n,(M\chi_{[-1,1]^n})^\delta).
\]
\item[{\rm (2)}] There exists a sufficiently large number $N\in{\mathbb N}$ such that
\begin{equation}\label{eq:210915-2}
|\langle f,\varphi\rangle|
\lesssim
\|f\|_{{\mathcal M}^p_1}\cdot p_N(\varphi)
\end{equation}
for all $f\in{\mathcal M}^p_1({\mathbb R}^n)$ and ${\mathcal S}({\mathbb R}^n)$.
In particular, the embedding
\[
{\mathcal M}^p_1({\mathbb R}^n)
\hookrightarrow{\mathcal S}'({\mathbb R}^n)
\]
holds.
\end{itemize}
\end{lemma}

\begin{proof}
The proof of (1) is similar to \cite[Proposition 285]{SDH20-1}.

For convenience, we give the proof of (2).
We estimate
\begin{align*}
|\langle f,\varphi\rangle|
&\le
\int_{{\mathbb R}^n}|f(x)\varphi(x)|\,{\rm d}x
\le
\int_{{\mathbb R}^n}\frac{|f(x)|}{(1+|x|)^n}p_n(\varphi)\,{\rm d}x\\
&\le
\left(
\int_{[-1,1]^n}|f(x)|\,{\rm d}x
+
\sum_{j=1}^\infty\frac1{(\sqrt{n}2^{j-1})^n}\int_{[-2^j,2^j]^n}|f(x)|\,{\rm d}x
\right)p_n(\varphi)\\
&\lesssim
\left(
\|f\|_{{\mathcal M}^p_1}
+
\frac{2^{\left(2-\frac1p\right)n}}{\sqrt{n}^n}
\sum_{j=1}^\infty2^{-\frac{jn}p}\|f\|_{{\mathcal M}^p_1}
\right)p_n(\varphi)\\
&=
\left(
1+\frac{2^{\left(2-\frac1p\right)n}}{\sqrt{n}^n}\frac{2^{-\frac np}}{1-2^{-\frac np}}
\right)
\|f\|_{{\mathcal M}^p_1}\cdot p_n(\varphi).
\end{align*}
Then, \eqref{eq:210915-2} is proved.
\end{proof}

Let $w$ be a locally integrable function.
We define the weighted $L^1$-norm $\|\cdot\|_{L^1(w)}$ by
\begin{equation*}
\|f\|_{L^1(w)}
\equiv
\int_{{\mathbb R}^n}|f(x)|w(x)\,{\rm d}x,
\quad f\in L^0({\mathbb R}^n),
\end{equation*}
and set
\begin{equation*}
H^1({\mathbb R}^n,w)
\equiv
\left\{
f\in{\mathcal S}'({\mathbb R}^n)
\,:\,
\|f\|_{H^1(w)}\equiv\left\|\sup_{t>0}|e^{t\Delta}f|\right\|_{L^1(w)}<\infty
\right\}.
\end{equation*}
We use the following atomic decomposition for $H^1({\mathbb R}^n,w)$ with the $A_1$-weight $w$.

\begin{theorem}[\cite{NaSa14}]\label{thm:210423-1}
Let $w$ be an $A_1$-weight, and let $f\in H^1({\mathbb R}^n,w)$.
Then there exists a triplet
$\{\lambda_j\}_{j=1}^\infty\subset[0,\infty)$,
$\{Q_j\}_{j=1}^\infty\subset{\mathcal Q}({\mathbb R}^n)$
and
$
\{a_j\}_{j=1}^\infty\subset L^\infty({\mathbb R}^n)
$
such that
$f=\sum_{j=1}^\infty\lambda_ja_j$ in ${\mathcal S}'({\mathbb R}^n)$
and that
\begin{equation*}
|a_j|\le\chi_{Q_j}, \quad
\int_{{\mathbb R}^n}a_j(x)\,{\rm d}x=0, \quad
\left\|
\sum_{j=1}^\infty\lambda_j\chi_{Q_j}
\right\|_{L^1(w)}
\lesssim
\|f\|_{H^1(w)}.
\end{equation*}
In particular, $H^1({\mathbb R}^n,w)$ embedded into $L^1({\mathbb R}^n,w)$.
\end{theorem}

Now, we start the proof of Proposition \ref{prop:200308-1}.

\begin{proof}[Proof of {\rm Proposition \ref{prop:200308-1}}]
(1)
By embedding for Morrey-Lorentz spaces (see Proposition \ref{prop:Mpqr-embedding}) and Lemma \ref{lem:211213-1} (2), we have
\[
f\in{\mathcal M}^p_{q,r}({\mathbb R}^n)
\hookrightarrow{\mathcal M}^p_1({\mathbb R}^n)
\hookrightarrow{\mathcal S}'({\mathbb R}^n).
\]
As is described in \cite[Proposition 2.7]{Duoandikoetxea01}, we have a pointwise estimate
\[
|e^{t\Delta}f| \le Mf.
\]
Since $M$ is shown to be bounded on the Morrey-Lorentz spaces ${\mathcal M}^p_{q,r}({\mathbb R}^n)$ (Proposition \ref{prop:190622-2}), we have $f\in H{\mathcal M}^p_{q,r}({\mathbb R}^n)$.

(2)
Let $f\in H{\mathcal M}^p_{q,r}({\mathbb R}^n)$.
Using Proposition \ref{prop:Mpqr-embedding} and Lemma \ref{lem:211213-1} (1), we have
\begin{equation*}
{\mathcal M}^p_{q,r}({\mathbb R}^n)
\hookrightarrow
{\mathcal M}^p_1({\mathbb R}^n)
\hookrightarrow
L^1({\mathbb R}^n,(M\chi_{[-1,1]^n})^\delta).
\end{equation*}
Since $(M\chi_{[-1,1]^n})^\delta$ is an $A_1$-weight (see \cite[Theorem 7.7]{Duoandikoetxea01}), by Theorem \ref{thm:210423-1},
\begin{equation*}
H{\mathcal M}^p_{q,r}({\mathbb R}^n)
\hookrightarrow
H^1({\mathbb R}^n,(M\chi_{[-1,1]^n})^\delta)
\hookrightarrow
L^1({\mathbb R}^n,(M\chi_{[-1,1]^n})^\delta).
\end{equation*}
Consequently, applying Lemma \ref{lem:210910-1} (1), we obtain
\begin{equation*}
H{\mathcal M}^p_{q,r}({\mathbb R}^n)
\hookrightarrow
{\mathcal M}^p_{q,r}({\mathbb R}^n).
\end{equation*}
\end{proof}

\subsection{Proof of Proposition \ref{prop:200308-2}}

The proof is similar to Hardy spaces with variable exponents \cite{CrWa14,NaSa12}.
We content ourselves with stating two fundamental estimates (\ref{eq:120918-5}) and (\ref{eq:120918-6}) below.

Suppose that we are given an integer $K \gg 1$.
We write
\begin{equation*}
M^*_{\rm heat} f(x)
\equiv 
\sup_{j \in {\mathbb Z}}
\left(\sup_{y \in {\mathbb R}^n}\frac{|e^{2^j\Delta}f(y)|}{(1+4^j|x-y|^2)^K}\right),
\quad x\in{\mathbb R}^n.
\end{equation*}

The next lemma stands for the pointwise estimate for $M^*_{\rm heat}$ in terms of the usual Hardy-Littlewood maximal operator $M$.

\begin{lemma}[{\rm \cite[Lemma 3.2]{NaSa12}, \cite[\S 4]{Sawano09}}]
For $0<\theta<1$, there exists $K_\theta$ so that for all $K \ge K_\theta$, we have
\begin{equation}\label{eq:120918-5}
M^*_{\rm heat} f(x)
\lesssim
M\left[\sup_{k \in {\mathbb Z}}|e^{2^k\Delta}f|^\theta\right](x)^{\frac{1}{\theta}},
\quad x\in{\mathbb R}^n
\end{equation}
for any $f\in{\mathcal S}'({\mathbb R}^n)$.
\end{lemma}

In the course of the proof of \cite[Theorem 3.3]{NaSa12}, it can be shown that
\begin{equation}\label{eq:120918-6}
{\mathcal M}f(x)
\sim
\sup_{\tau \in {\mathcal F}_N, \, j \in {\mathbb Z}}|\tau^j*f(x)|
\lesssim
M^*_{\rm heat} f(x)
\end{equation}
once we fix an integer $K \gg 1$ and $N \gg 1$.

Applying Proposition \ref{prop:190622-2} with the fundamental pointwise estimates (\ref{eq:120918-5}) and (\ref{eq:120918-6}), Proposition \ref{prop:200308-2} can be proved with ease.
We omit the details.

\section{Proof of Theorems \ref{thm:200308-1}, \ref{thm:200308-3} and \ref{thm:200929-1}}\label{s4}

\subsection{Proof of Theorem \ref{thm:200308-1}}

We employ the argument from the proof of Theorem 1.1 in \cite{IST14}.

By the decomposition of $Q_j$, we may assume that each $Q_j$ is a dyadic cube.
We may assume that there exists $N \in {\mathbb N}$ such that $\lambda_j=0$ whenever $j \ge N$.
In addition, let us assume that the $a_j$'s are non-negative.
By the embedding $\ell^v({\mathbb N})\hookrightarrow\ell^1({\mathbb N})$ and the duality argument, we note that
\begin{align*}
\|f\|_{{\mathcal M}^p_{q,r}}^v
\le
\left\|\sum_{j=1}^\infty|\lambda_ja_j|^v\right\|_{{\mathcal M}^{\tilde{p}}_{\tilde{q},\tilde{r}}}
\sim
\sup
\left\{
\int_{{\mathbb R}^n}\sum_{j=1}^\infty|\lambda_ja_j(x)|^v|g(x)|\,{\rm d}x
:
\|g\|_{{\mathcal H}^{\tilde{p}'}_{\tilde{q}',\tilde{r}'}}=1
\right\},
\end{align*}
where we set $\tilde{p}:=p/v$, $\tilde{q}:=q/v$ and $\tilde{r}:=r/v$.
Then, we may assume that the $a_j$'s are non-negative and $g$ is a non-negative $(\tilde{p}',\tilde{q}',\tilde{r}')$-block with associated dyadic cube $Q$.
Then, we show that
\begin{equation}\label{eq:200308-6}
\int_{{\mathbb R}^n}\sum_{j=1}^\infty(\lambda_ja_j(x))^vg(x)\,{\rm d}x
\lesssim_{p,q,r,s,t,v}
\left\|
\left(\sum_{j=1}^{\infty}(\lambda_j\chi_{Q_j})^v\right)^{\frac1v}
\right\|_{{\mathcal M}^p_{q,r}}^v.
\end{equation}

Assume first that each $Q_j$ contains $Q$ as a proper subset.
If we group the $j$'s such that all $Q_j$ are identical,
we can assume that each $Q_j$ is a $j$-th parent of $Q$ for each $j\in{\mathbb N}$.
Then, by the H\"older inequality for Lorentz spaces (see Lemma \ref{lem:Lpq-Holder}),
\begin{align*}
\int_{{\mathbb R}^n}\sum_{j=1}^\infty(\lambda_ja_j(x))^vg(x)\,{\rm d}x
&=
\sum_{j=1}^\infty\lambda_j^v\int_{Q}a_j(x)^vg(x)\,{\rm d}x\\
&\lesssim
\sum_{j=1}^\infty
\lambda_j^v\|a_j\chi_Q\|_{L^{t,\infty}}^v\|g\|_{L^{\tilde{q}',\tilde{r}'}}|Q|^{\frac vq-\frac vt}\\
&\le
\sum_{j=1}^\infty\lambda_j^v\|a_j\|_{{\rm W}{\mathcal M}^s_t}^v
|Q|^{-\frac vs+\frac vt}|Q|^{\frac1{\tilde{q}'}-\frac1{\tilde{p}'}}|Q|^{\frac vq-\frac vt}\\
&=
\sum_{j=1}^\infty\lambda_j^v|Q_j|^{\frac vs}|Q|^{\frac vp-\frac vs}.
\end{align*}
Note that by Proposition \ref{prop:Mpqr-indicator}, for each $J\in{\mathbb N}$,
\begin{align*}
\left\|
\left(\sum_{k=1}^\infty(\lambda_k\chi_{Q_k})^v\right)^{\frac1v}
\right\|_{{\mathcal M}^p_{q,r}}
\ge
\lambda_J\left\|\chi_{Q_J}\right\|_{{\mathcal M}^p_{q,r}}
\sim
\lambda_J|Q_J|^{\frac1p}.
\end{align*}
Consequently, it follows from the condition $p<s$ that
\begin{align*}
\int_{{\mathbb R}^n}\sum_{j=1}^\infty(\lambda_ja_j(x))^vg(x)\,{\rm d}x
&\lesssim
\sum_{j=1}^\infty
|Q_j|^{\frac vs-\frac vp}|Q|^{\frac vp-\frac vs}
\cdot
\left\|
\left(\sum_{k=1}^\infty(\lambda_k\chi_{Q_k})^v\right)^{\frac1v}
\right\|_{{\mathcal M}^p_{q,r}}^v\\
&\lesssim
\left\|
\left(\sum_{k=1}^\infty(\lambda_k\chi_{Q_k})^v\right)^{\frac1v}
\right\|_{{\mathcal M}^p_{q,r}}^v.
\end{align*}
Conversely, assume that $Q$ contains each $Q_j$.
Then, again by the H\"older inequality (see Lemma \ref{lem:Lpq-Holder}),
\begin{align*}
\int_{{\mathbb R}^n}\sum_{j=1}^\infty(\lambda_ja_j(x))^vg(x)\,{\rm d}x
&=
\sum_{j=1}^\infty\lambda_j^v\int_{Q_j}a_j(x)^vg(x)\,{\rm d}x
\lesssim
\sum_{j=1}^\infty\lambda_j^v\|a_j\|_{L^{t,\infty}}^v\|g\chi_{Q_j}\|_{L^{\tilde{t}',1}}\\
&\le
\sum_{j=1}^\infty\lambda_j\|a_j\|_{{\rm W}{\mathcal M}^s_t}^v
|Q_j|^{-\frac vs+\frac vt}\|g\chi_{Q_j}\|_{L^{\tilde{t}',1}}\\
&\le
\sum_{j=1}^\infty\lambda_j^v|Q_j|^{\frac vt}\|g\chi_{Q_j}\|_{L^{\tilde{t}',1}},
\end{align*}
where $\tilde{t}:=t/v$.
Thus, in terms of the maximal operator $M^{(\tilde{t}',1)}$,
we obtain
\begin{align*}
\int_{{\mathbb R}^n}\sum_{j=1}^\infty(\lambda_ja_j(x))^vg(x)\,{\rm d}x
&\le
\sum_{j=1}^\infty\lambda_j^v|Q_j|\cdot\inf_{y \in Q_j}M^{(\tilde{t}',1)}g(y)\\
&\le
\int_{{\mathbb R}^n}
\left(\sum_{j=1}^\infty(\lambda_j\chi_{Q_j}(y))^v\right)\chi_Q(y)M^{(\tilde{t}',1)}g(y)
\,{\rm d}y.
\end{align*}
Hence, we obtain (\ref{eq:200308-6}) by the H\"older inequality (see Lemma \ref{lem:Lpq-Holder}) and the $L^{\tilde{q}',\tilde{r}'}({\mathbb R}^n)$-boundedness of the maximal operator $M^{(\tilde{t}',1)}$ (see Proposition \ref{prop:190620-1}).

It remains to check the convergence of the sum.
Here, when $r<\infty$, by the estimate of \eqref{eq:200308-1.3}, the Lebesgue convergence theorem yields
\begin{equation*}
\left\|
\left(\sum_{j=1}^\infty\lambda_ja_j-\sum_{j=1}^J\lambda_ja_j\right)\chi_R
\right\|_{L^{q,r}}
\le
\left\|\left(\sum_{j=J+1}^\infty|\lambda_ja_j|\right)\chi_R\right\|_{L^{q,r}}
\to0
\end{equation*}
as $J\to\infty$ for each $R\in{\mathcal Q}({\mathbb R}^n)$.
Namely, $f=\sum_{j=1}^{\infty} \lambda_j a_j$ converges in $L_{\rm loc}^{q,r}({\mathbb R}^n)$.
The case of $p=q$ and $r<\infty$ can be also dealt with by a similar approach.

\subsection{Lemmas for the proofs of Theorems \ref{thm:200929-1} and \ref{thm:200308-4}}\label{ss:keylem}

To prove Theorems \ref{thm:200929-1} and \ref{thm:200308-4}, we gave some estimate for atoms belonging to the critical cases ${\mathcal M}^s_1({\mathbb R}^n)$.

\begin{lemma}\label{lem:210803-11}
Let $1\le s<\infty$, $K\in{\mathbb N}_0$, and $Q\in{\mathcal Q}({\mathbb R}^n)$.
Assume that $a\in{\mathcal M}^s_1({\mathbb R}^n)\cap{\mathcal P}_K({\mathbb R}^n)^\perp$ satisfies
\begin{equation}\label{eq:210908-1}
{\rm supp}(a)\subset Q, \quad
\|a\|_{{\mathcal M}^s_1}\le|Q|^{\frac1s}.
\end{equation}
Then, for all $\varphi\in{\mathcal S}({\mathbb R}^n)$ and $N>0$,
\begin{equation*}
\left|\int_{{\mathbb R}^n}a(x)\varphi(x)\,{\rm d}x\right|
\lesssim
\ell(Q)^{n+K+1}\sup_{y\in Q}\frac1{1+|y|^N}.
\end{equation*}
Here, the implicit constant in $\lesssim$ depends on $\varphi$.
\end{lemma}

\begin{proof}
By the mean-value theorem, there exists $\theta\in(0,1)$ depending on $x$, $K$, $Q$, and $\varphi$ such that
\begin{align*}
\int_{{\mathbb R}^n}a(x)\varphi(x)\,{\rm d}x
&=
\int_{{\mathbb R}^n}
a(x)
\left(
\varphi(x)-\sum_{|\alpha|\le K}\frac1{\alpha!}\partial^\alpha\varphi(c(Q))(x-c(Q))^\alpha
\right)
\,{\rm d}x\\
&=
\int_{{\mathbb R}^n}
a(x)
\sum_{|\beta|=K+1}
\frac1{\beta!}\partial^\beta\varphi((1-\theta)x+\theta c(Q))(x-c(Q))^\beta
\,{\rm d}x.
\end{align*}
Then, from \eqref{eq:210908-1}
\begin{align*}
\left|\int_{{\mathbb R}^n}a(x)\varphi(x)\,{\rm d}x\right|
&\lesssim
\ell(Q)^{K+1}\sup_{y\in Q}\frac1{1+|y|^N}\int_Q|a(x)|\,{\rm d}x\\
&\le
\ell(Q)^{K+1}\sup_{y\in Q}\frac1{1+|y|^K}
\|a\|_{{\mathcal M}^s_1}\cdot|Q|^{-\frac1s+1}\\
&\lesssim
\ell(Q)^{n+K+1}\sup_{y\in Q}\frac1{1+|y|^N},
\end{align*}
as desired.
\end{proof}

\begin{lemma}\label{lem:210924-1}
Let $0<q\le p<\infty$, $1\le s<\infty$, and $K\in{\mathbb N}_0$.
Assume that $\{\lambda_Q\}_{Q\in{\mathcal D}({\mathbb R}^n)}\subset[0,\infty)$ and $\{a_Q\}_{Q\in{\mathcal D}({\mathbb R}^n)}\subset{\mathcal M}^s_1({\mathbb R}^n)\cap{\mathcal P}_K({\mathbb R}^n)^\perp$ satisfy
\[
{\rm supp}(a_Q)\subset3Q, \quad
\|a\|_{{\mathcal M}^s_1}\le|Q|^{\frac1s}, \quad
\left\|\sum_{Q\in{\mathcal D}({\mathbb R}^n)}\lambda_Q\chi_Q\right\|_{{\mathcal M}^p_q}
<\infty.
\]
If $q\le1$ and $n+K+1>n/q$, then
\begin{equation}\label{eq:210924-1}
\sum_{m=1}^\infty\sum_{Q\in{\mathcal D}_m({\mathbb R}^n)}
\lambda_Q|\langle a_Q,\varphi\rangle|
\lesssim
\left\|\sum_{Q\in{\mathcal D}({\mathbb R}^n)}\lambda_Q\chi_Q\right\|_{{\mathcal M}^p_q}.
\end{equation}
\end{lemma}

\begin{proof}
Fix $m\ge1$.
To prove \eqref{eq:210924-1}, we use the fact that for each $\tilde{m}\in{\mathbb Z}^n$,
\begin{align*}
\left\|
\sum_{Q\in{\mathcal D}({\mathbb R}^n)}\lambda_Q\chi_Q
\right\|_{{\mathcal M}^p_q}
&\gtrsim
\left\|\left(
\sum_{Q\in{\mathcal D}_m({\mathbb R}^n),\,|c(Q)-\tilde{m}|\le n}\lambda_Q\chi_Q
\right)\chi_{\tilde{m}+[-n,n]^n}\right\|_{L^q}\\
&=
2^{-\frac{mn}q}
\left(
\sum_{Q\in{\mathcal D}_m({\mathbb R}^n),\,|c(Q)-\tilde{m}|\le n}
\lambda_Q^q
\right)^{\frac1q}\\
&\ge
2^{-\frac{mn}q}
\sum_{Q\in{\mathcal D}_m({\mathbb R}^n),\,|c(Q)-\tilde{m}|\le n}\lambda_Q.
\end{align*}
In particular, for all $R\in{\mathcal D}_m({\mathbb R}^n)$,
\begin{equation*}
\left\|
\sum_{Q\in{\mathcal D}({\mathbb R}^n)}\lambda_Q\chi_Q
\right\|_{{\mathcal M}^p_q}
\gtrsim
2^{-\frac{mn}q}\lambda_R.
\end{equation*}
We remark that for each $m\ge1$ and $\tilde{m}\in{\mathbb Z}^n$,
\[
\sharp\{Q\in{\mathcal D}_m({\mathbb R}^n)\,:\,\overline{3Q}\ni0\}
=4^n.
\]
It follows from Lemma \ref{lem:210803-11} that
\begin{align}\label{eq:210923-1}
\begin{split}
\sum_{Q\in{\mathcal D}_m({\mathbb R}^n),\,\overline{3Q}\ni0}
\lambda_Q|\langle a_Q,\varphi\rangle|
&\lesssim
2^{-m(n+K+1)}
\sum_{Q\in{\mathcal D}_m({\mathbb R}^n),\,\overline{3Q}\ni0}\lambda_Q\\
&\lesssim
2^{-m\left(n+K+1-\frac nq\right)}
\left\|
\sum_{Q\in{\mathcal D}({\mathbb R}^n)}\lambda_Q\chi_Q
\right\|_{{\mathcal M}^p_q}.
\end{split}
\end{align}
In addition, setting
\[
{\mathcal D}_{m,\tilde{m}}({\mathbb R}^n)
:=
\{Q\in{\mathcal D}_m({\mathbb R}^n)\,:\,|c(Q)-\tilde{m}|\le n\}
\]
for each $m\ge1$ and $\tilde{m}\in{\mathbb Z}^n$, we have
\[
{\mathcal D}_m({\mathbb R}^n)
=
\bigcup_{\tilde{m}\in{\mathbb Z}^n}{\mathcal D}_{m,\tilde{m}}({\mathbb R}^n).
\]
Then, there exists a mapping $\iota_m:{\mathcal D}_m({\mathbb R}^n)\to{\mathbb Z}^n$ such that $Q\in{\mathcal D}_{m,\iota_m(Q)}({\mathbb R}^n)$; therefore
\begin{align}\label{eq:210923-2}
\begin{split}
&\sum_{Q\in{\mathcal D}_m({\mathbb R}^n),\,\overline{3Q}\not\ni0}
\lambda_Q|\langle a_Q,\varphi\rangle|\\
&\lesssim
2^{-m(n+K+1)}
\sum_{\tilde{m}\in{\mathbb Z}^n}
\sum_{Q\in{\mathcal D}_m({\mathbb R}^n),\,\iota_m(Q)=\tilde{m}}
\lambda_Q\cdot\sup_{y\in3Q}\frac1{1+|y|^{n+1}}\\
&\lesssim
2^{-m(n+K+1)}
\sum_{\tilde{m}\in{\mathbb Z}^n}\frac1{1+|\tilde{m}|^{n+1}}
\sum_{Q\in{\mathcal D}_{m,\tilde{m}}({\mathbb R}^n)}\lambda_Q\\
&\lesssim
2^{-m\left(n+K+1-\frac nq\right)}
\left\|
\sum_{Q\in{\mathcal D}({\mathbb R}^n)}\lambda_Q\chi_Q
\right\|_{{\mathcal M}^p_q}.
\end{split}
\end{align}
Because
\[
n+K+1>\frac nq,
\]
then we obtain the desired result.
\end{proof}

\subsection{Proofs of Theorems \ref{thm:200308-3} and \ref{thm:200929-1}: convergence of $f=\sum_{j=1}^\infty\lambda_ja_j$}\label{ss:Proof of Thm1}

First, we prove the convergence of
\begin{equation*}
f=\sum_{j=1}^\infty\lambda_ja_j
\quad \text{in ${\mathcal S}'({\mathbb R}^n)$.}
\end{equation*}

We start with an important reduction.
For each $J\in{\mathbb N}$, we take any cube $Q(J) \in {\mathcal D}({\mathbb R}^n)$ with mimimal volume such that $Q_J\subset3Q(J)$, and we set
\[
{\mathcal E}_Q:=\{j \in {\mathbb N}\,:\,Q=Q(j)\}
\]
and
\begin{equation*}
\lambda_Q:=\sum_{j\in{\mathcal E}_Q}\lambda_j, \quad
a_Q
:=
\begin{cases}
0, & \lambda_Q=0,\\
\displaystyle
\frac1{\lambda_Q}\sum_{j\in{\mathcal E}_Q}\lambda_ja_j, & \lambda_Q\ne0.
\end{cases}
\end{equation*}
Note that $\{{\mathcal E}_Q\}_{Q\in{\mathcal D}({\mathbb R}^n)}$ is pairwise disjoint.
Then, $\{a_Q\}_{Q\in{\mathcal D}({\mathbb R}^n)}$ and $\{\lambda_Q\}_{Q\in{\mathcal D}({\mathbb R}^n)}$ satisfy
\begin{align*}
\|a_Q\|_{{\mathcal M}^s_1}
\le
\frac1{\lambda_Q}\sum_{j\in{\mathcal E}_Q}\lambda_j\|a_j\|_{{\mathcal M}^s_1}
\le
\frac1{\lambda_Q}\sum_{j\in{\mathcal E}_Q}\lambda_j|Q_j|^{\frac1s}
\le
|3Q|^{\frac1s}.
\end{align*}
Taking $\theta\in(1/v,\infty)$, by the fact that $\chi_{Q(J)}\lesssim_nM\chi_{Q_J}$ for $J\in{\mathbb N}$, we have
\begin{align*}
\left\|
\left(\sum_{Q\in{\mathcal D}({\mathbb R}^n)}(\lambda_Q\chi_Q)^v\right)^{\frac1v}
\right\|_{{\mathcal M}^p_{q,r}}
&\lesssim
\left\|
\left(
\sum_{Q\in{\mathcal D}({\mathbb R}^n)}
\left(\sum_{j\in{\mathcal E}_Q}\lambda_j(M\chi_{Q_j})^\theta\right)^v
\right)^{\frac1v}
\right\|_{{\mathcal M}^p_{q,r}}\\
&\le
\left\|
\left(
\sum_{Q\in{\mathcal D}({\mathbb R}^n)}\sum_{j\in{\mathcal E}_Q}
\left(M\left[\lambda_j^{\frac1\theta}\chi_{Q_j}\right]\right)^{\theta v}
\right)^{\frac1{\theta v}}
\right\|_{{\mathcal M}^{\theta p}_{\theta q,\theta r}}^\theta.
\end{align*}
Then, by Proposition \ref{prop:200324-1},
\begin{align*}
\left\|
\left(\sum_{Q\in{\mathcal D}({\mathbb R}^n)}(\lambda_Q\chi_Q)^v\right)^{\frac1v}
\right\|_{{\mathcal M}^p_{q,r}}
&\lesssim
\left\|
\left(
\sum_{Q\in{\mathcal D}({\mathbb R}^n)}\sum_{j\in{\mathcal E}_Q}
(\lambda_j\chi_{Q_j})^v
\right)^{\frac1{\theta v}}
\right\|_{{\mathcal M}^{\theta p}_{\theta q,\theta r}}^\theta\\
&=
\left\|
\left(
\sum_{j=1}^\infty(\lambda_j\chi_{Q_j})^v
\right)^{\frac1{\theta v}}
\right\|_{{\mathcal M}^{\theta p}_{\theta q,\theta r}}^\theta
<\infty.
\end{align*}
Hence, we may assume
\begin{equation*}
\{a_Q\}_{Q\in{\mathcal D}}
\subset
\begin{cases}
{\rm W}{\mathcal M}^s_t({\mathbb R}^n)\cap{\mathcal P}_{d_v}({\mathbb R}^n)^\perp,
& 1<t\le s<\infty,\\
{\mathcal M}^s_1({\mathbb R}^n)\cap{\mathcal P}_{d_v}({\mathbb R}^n)^\perp,
& 1=t\le s<\infty,
\end{cases}
\quad
\{\lambda_Q\}_{Q\in{\mathcal D}} \subset[0,\infty),
\end{equation*}
with ${\rm supp}(a_Q)\subset3Q$ for $Q\in{\mathcal D}({\mathbb R}^n)$
instead of
\begin{equation*}
\{a_j\}_{j=1}^\infty
\subset
\begin{cases}
{\rm W}{\mathcal M}^s_t({\mathbb R}^n)\cap{\mathcal P}_{d_v}({\mathbb R}^n)^\perp,
& 1<t\le s<\infty,\\
{\mathcal M}^s_1({\mathbb R}^n)\cap{\mathcal P}_{d_v}({\mathbb R}^n)^\perp,
& 1=t\le s<\infty,
\end{cases}
\quad
\{\lambda_j\}_{j=1}^\infty \subset[0,\infty).
\end{equation*}
Then, it suffices to show that
\begin{equation}\label{eq:210803-1}
\sum_{Q\in{\mathcal D}({\mathbb R}^n)}\lambda_Q|\langle a_Q,\varphi\rangle|
=
\sum_{m=-\infty}^0
\sum_{Q\in{\mathcal D}_m({\mathbb R}^n)}\lambda_Q|\langle a_Q,\varphi\rangle|
+
\sum_{m=1}^\infty
\sum_{Q\in{\mathcal D}_m({\mathbb R}^n)}\lambda_Q|\langle a_Q,\varphi\rangle|
<\infty.
\end{equation}

First, we estimate the first part of \eqref{eq:210803-1}.
Fix $m\le0$.
For each $Q\in{\mathcal D}_m({\mathbb R}^n)$, $\overline{3Q}\not\ni0$ implies that $|y|\ge\ell(Q)$ for all $y\in3Q$, and then
\begin{align*}
|\langle a_Q,\varphi\rangle|
&\lesssim
\int_{3Q}|a_Q(x)|\,{\rm d}x\sup_{y\in 3Q}\frac1{1+|y|^{2n+1-\frac ns}}\\
&\lesssim
\|a_Q\|_{{\mathcal M}^s_1}\cdot|Q|^{-\frac1s+1}\sup_{y\in 3Q}\frac1{1+|y|^{2n+1-\frac ns}}
\lesssim
|Q|^{\frac1s}\sup_{y\in 3Q}\frac1{1+|y|^{n+1}}.
\end{align*}
It follows that
\begin{align*}
\sum_{Q\in{\mathcal D}_m({\mathbb R}^n),\,\overline{3Q}\not\ni0}
\lambda_Q|\langle a_Q,\varphi\rangle|
&\lesssim
\sum_{Q\in{\mathcal D}_m({\mathbb R}^n),\,\overline{3Q}\not\ni0}
\lambda_Q\cdot|Q|^{\frac1s}\sup_{y\in 3Q}\frac1{1+|y|^{n+1}}\\
&\lesssim
2^{-\frac{mn}s}
\sum_{\tilde{m}\in{\mathbb Z}^n}\frac1{1+|\tilde{m}|^{n+1}}
\sum_{\substack{Q\in{\mathcal D}_m({\mathbb R}^n),\,\overline{3Q}\not\ni0\\ |c(Q)-\tilde{m}|\le n}}
\lambda_Q.
\end{align*}
Meanwhile, if $\overline{3Q}\ni0$, by Lemma \ref{lem:211213-1} (2),
\begin{align*}
|\langle a_Q,\varphi\rangle|
\lesssim_\varphi
\|a_Q\|_{{\mathcal M}^s_1}
\lesssim
|Q|^{\frac1s}.
\end{align*}
Thus,
\begin{align*}
\sum_{Q\in{\mathcal D}_m({\mathbb R}^n),\,\overline{3Q}\ni0}
\lambda_Q|\langle a_Q,\varphi\rangle|
\lesssim
\sum_{Q\in{\mathcal D}_m({\mathbb R}^n),\,\overline{3Q}\not\ni0}
\lambda_Q\cdot|Q|^{\frac1s}
\le
4^n2^{-\frac{mn}s+\frac{mn}p}
\sup_{Q\in{\mathcal D}({\mathbb R}^n)}\lambda_Q|Q|^{\frac1p}.
\end{align*}
Note that for each $R\in{\mathcal D}({\mathbb R}^n)$,
\begin{align*}
\left\|
\left(\sum_{Q\in{\mathcal D}({\mathbb R}^n)}(\lambda_Q\chi_Q)^v\right)^{\frac1v}
\right\|_{{\mathcal M}^p_{q,r}}
\ge
\|\lambda_R\chi_R\|_{{\mathcal M}^p_{q,r}}
=
\left(\frac qr\right)^{\frac1r}\lambda_R|R|^{\frac1p}
\end{align*}
by Proposition \ref{prop:Mpqr-indicator}.
We conclude that
\begin{align}\label{eq:210808-1}
\begin{split}
\sum_{m=-\infty}^0
\sum_{Q\in{\mathcal D}_m({\mathbb R}^n)}\lambda_Q|\langle a_Q,\varphi\rangle|
&\lesssim
\sum_{m=-\infty}^02^{-\frac{mn}s+\frac{mn}p}
\left\|
\left(\sum_{Q\in{\mathcal D}({\mathbb R}^n)}(\lambda_Q\chi_Q)^v\right)^{\frac1v}
\right\|_{{\mathcal M}^p_{q,r}}\\
&\lesssim
\left\|
\left(\sum_{Q\in{\mathcal D}({\mathbb R}^n)}(\lambda_Q\chi_Q)^v\right)^{\frac1v}
\right\|_{{\mathcal M}^p_{q,r}}.
\end{split}
\end{align}

Note that
\begin{align*}
n+d_v+1-\frac n{q_0}
>
n+\left(\frac nv-n-1\right)+1-\frac n{q_0}
=
\frac nv-\frac n{q_0}
\ge0.
\end{align*}
Thus, there exists $\varepsilon\in(0,q_0)$ such that
\[
n+d_v+1-\frac n{q_0-\varepsilon}>0,
\]
where $q_0:=\min(1,q)$.
Hence, by Lemma \ref{lem:210924-1}, the second part of \eqref{eq:210803-1} can be estimated as follows:
\begin{align}\label{eq:210808-2}
\begin{split}
\sum_{m=1}^\infty
\sum_{Q\in{\mathcal D}_m({\mathbb R}^n)}\lambda_Q|\langle a_Q,\varphi\rangle|
\lesssim
\left\|
\sum_{Q\in{\mathcal D}({\mathbb R}^n)}\lambda_Q\chi_Q
\right\|_{{\mathcal M}^p_{q_0-\varepsilon}}
\lesssim
\left\|
\left(\sum_{Q\in{\mathcal D}({\mathbb R}^n)}(\lambda_Q\chi_Q)^v\right)^{\frac1v}
\right\|_{{\mathcal M}^p_{q,r}},
\end{split}
\end{align}
where in the last inequality, we used the embedding
$
{\mathcal M}^p_{q,r}({\mathbb R}^n)\hookrightarrow
{\mathcal M}^p_{q_0-\varepsilon,q_0-\varepsilon}({\mathbb R}^n)
=
{\mathcal M}^p_{q_0-\varepsilon}({\mathbb R}^n)
$
(see Proposition \ref{prop:Mpqr-embedding}).
Combining these estimates \eqref{eq:210808-1} and \eqref{eq:210808-2}, we finish the proof of \eqref{eq:210803-1}.

\subsection{Proofs of Theorems \ref{thm:200308-3} and \ref{thm:200929-1}: \eqref{eq:200308-3.3} and \eqref{eq:200929-1.3}}\label{ss: proof atom}

To prove the estimates of \eqref{eq:200308-3.3} and \eqref{eq:200929-1.3} in Theorems \ref{thm:200308-3} and \ref{thm:200929-1}, respectively, we use the following lemma, whose proof is similar to that of Lemma \ref{lem:210803-11}.

\begin{lemma}[{\cite[(5.2)]{NaSa12}}]\label{lem:210723-1}
If $\{a_j\}_{j=1}^\infty$ satisfies the same assumptions as in Theorems {\rm \ref{thm:200308-3}} and {\rm \ref{thm:200929-1}}, then
\begin{equation*}
{\mathcal M}a_j(x)
\lesssim
\chi_{3Q_j}(x)Ma_j(x)+M\chi_{Q_j}(x)^{\frac{n+d_v+1}{n}},
\quad x\in{\mathbb R}^n.
\end{equation*}
\end{lemma}

Let us show Theorem \ref{thm:200308-3}.
Using Proposition \ref{prop:200308-2} and Lemma \ref{lem:210723-1}, we have
\begin{align*}
\|f\|_{H{\mathcal M}^p_{q,r}}
&\sim
\|{\mathcal M}f\|_{{\mathcal M}^p_{q,r}}
\le
\left\|\sum_{j=1}^\infty \lambda_j {\mathcal M}a_j\right\|_{{\mathcal M}^p_{q,r}}\\
&\lesssim
\left\|
\sum_{j=1}^\infty\lambda_j
\left(\chi_{3Q_j}Ma_j+(M\chi_{Q_j})^{\frac{n+d_v+1}{n}}\right)
\right\|_{{\mathcal M}^p_{q,r}}\\
&\lesssim
\left\|\sum_{j=1}^\infty\lambda_j\chi_{3Q_j}Ma_j\right\|_{{\mathcal M}^p_{q,r}}
+
\left\|
\sum_{j=1}^\infty\lambda_j(M\chi_{Q_j})^{\frac{n+d_v+1}{n}}
\right\|_{{\mathcal M}^p_{q,r}}
=:
I_1+I_2.
\end{align*}

First, we consider $I_1$.
We note that for each $j\in{\mathbb N}$, owing to the ${\rm W}{\mathcal M}^s_t({\mathbb R}^n)={\mathcal M}^s_{t,\infty}({\mathbb R}^n)$-boundedness of $M$ (see Proposition \ref{prop:190622-2}),
by applying Proposition \ref{prop:200324-1} and Theorem \ref{thm:200308-1} and using the fact that $\chi_{3Q_j}\le3^nM\chi_{Q_j}$ for each $j\in{\mathbb N}$, we have
\begin{align*}
I_1
&\lesssim
\left\|
\left(\sum_{j=1}^\infty(\lambda_j\chi_{3Q_j})^v\right)^{\frac1v}
\right\|_{{\mathcal M}^p_{q,r}}
\lesssim
\left\|
\left(\sum_{j=1}^\infty\lambda_j^v(M\chi_{Q_j})^2\right)^{\frac1v}
\right\|_{{\mathcal M}^p_{q,r}}\\
&\lesssim
\left\|
\left(\sum_{j=1}^\infty(\lambda_j\chi_{Q_j})^v\right)^{\frac1v}
\right\|_{{\mathcal M}^p_{q,r}}.
\end{align*}

Next, we consider $I_2$. 
Set
\begin{equation*}
P:=\frac{n+d_v+1}np, \quad
Q:=\frac{n+d_v+1}nq, \quad
\text{and} \quad
R:=\frac{n+d_v+1}nr.
\end{equation*}
Then, by Proposition \ref{prop:200324-1} and the embedding $\ell^v({\mathbb N})\hookrightarrow\ell^1({\mathbb N})$, we obtain
\begin{align*}
I_2
&=
\left\|
\left[
\sum_{j=1}^\infty\lambda_j(M\chi_{Q_j})^{\frac{n+d_v+1}n}
\right]^{\frac n{n+d_v+1}}
\right\|_{{\mathcal M}^P_{Q,R}}^{\frac{n+d_v+1}n}
\lesssim
\left\|\sum_{j=1}^\infty\lambda_j \chi_{Q_j}\right\|_{{\mathcal M}^p_{q,r}}\\
&\le
\left\|
\left(\sum_{j=1}^\infty(\lambda_j \chi_{Q_j})^v\right)^{\frac1v}
\right\|_{{\mathcal M}^p_{q,r}}.
\end{align*}
Thus, we obtain the desired result.\\

Similarly, because $M$ satisfies the weak-type boundedness on ${\mathcal M}^s_1({\mathbb R}^n)$ (see Proposition \ref{prop:200929-1}) in the estimate of $I_1$, we can prove Theorem \ref{thm:200929-1}.

\section{Proof of Theorem \ref{thm:200308-4}}\label{s5}

To prove Theorem \ref{thm:200308-4}, we use a new approach provided in \cite[Subsection 4.3]{NOSS20} and the following lemma, as given in \cite[Exercise 3.34]{Sawano18}.

\begin{lemma}\label{lem:210407-1}
Let $0<q\le p<\infty$, $0<r\le\infty$, $K\in{\mathbb N}$, and $0<v<\infty$, and let $f\in H{\mathcal M}^p_{q,r}({\mathbb R}^n)\cap L_{\rm loc}^1({\mathbb R}^n)$.
Then, we can find $\{a_j\}_{j=1}^\infty\subset L^\infty({\mathbb R}^n)\cap{\mathcal P}_K^\perp({\mathbb R}^n)$ and a sequence $\{Q_j\}_{j=1}^\infty$ of cubes such that
\begin{itemize}
\item[{\rm (1)}] ${\rm supp}(a_j)\subset Q_j$,
\item[{\rm (2)}] $\displaystyle f=\sum_{j=1}^\infty a_j$ in ${\mathcal S}'({\mathbb R}^n)$, and
\item[{\rm (3)}] $\displaystyle\left(\sum_{j=1}^\infty(\|a_j\|_{L^\infty}\chi_{Q_j})^v\right)^{\frac1v}\lesssim{\mathcal M}f$.
\end{itemize}
\end{lemma}

\begin{proof}[Proof of Theorem {\rm \ref{thm:200308-4}}]
It suffices to prove the case of $v=1$; the case of $v>0$ can be proved similarly.
Let $f \in H{\mathcal M}^p_{q,r}({\mathbb R}^n)$.
Fix $t>0$.
Because ${\mathcal D}({\mathbb R}^n)$ is a countable set, applying Lemma \ref{lem:210407-1} to $e^{t\Delta}f\in H{\mathcal M}^p_{q,r}({\mathbb R}^n)\cap L_{\rm loc}^1({\mathbb R}^n)$ for
\[
\{3Q\}_{Q\in{\mathcal D}({\mathbb R}^n)}, \quad
\{\lambda^t_Q\}_{Q\in{\mathcal D}({\mathbb R}^n)}, \quad
\{\lambda_Qa^t_Q\}_{Q\in{\mathcal D}({\mathbb R}^n)}
\]
instead of
\[
\{Q_j\}_{j=1}^\infty, \quad
\{\|a_j\|_{L^\infty}\}_{j=1}^\infty, \quad
\{a_j\}_{j=1}^\infty,
\]
respectively, we can consider the decomposition
$
e^{t\Delta}f=\sum_{Q\in{\mathcal D}({\mathbb R}^n)}\lambda^t_Q a^t_Q 
$
in the topology of ${\mathcal S}'({\mathbb R}^n)$, where $a^t_Q \in {\mathcal P}_K^\perp({\mathbb R}^n)$, $\lambda^t_Q \ge 0$, and
\[
|a^t_Q| \le \chi_{3Q}, \quad
\left\|\sum_{Q\in{\mathcal D}({\mathbb R}^n)}\lambda^t_Q\chi_{3Q}\right\|_{{\mathcal M}^p_{q,r}} 
\lesssim
\|{\mathcal M}[e^{t\Delta}f]\|_{{\mathcal M}^p_{q,r}}
\lesssim
\|{\mathcal M}f\|_{{\mathcal M}^p_{q,r}}.
\]
By the weak-* compactness of the unit ball of $L^\infty({\mathbb R}^n)$, there exists a sequence $\{t_l\}_{l=1}^\infty$ that converges to $0$ such that both $\lambda_Q=\lim_{l\to\infty}\lambda^{t_l}_Q$ and $a_Q=\lim_{l\to\infty}a^{t_l}_Q$ exist for all $Q \in {\mathcal D}$ in the sense that
\[
\lim_{l\to\infty}\int_{{\mathbb R}^n}a^{t_l}_Q(x)\varphi(x)\,{\rm d}x
=
\int_{{\mathbb R}^n}a_Q(x)\varphi(x)\,{\rm d}x
\]
for all $\varphi \in L^1({\mathbb R}^n)$.
We claim that $f=\sum_{Q\in{\mathcal D}({\mathbb R}^n)}\lambda_Qa_Q$ in the topology of ${\mathcal S}'({\mathbb R}^n)$.
Let $\varphi \in {\mathcal S}(\mathbb{R}^n)$ be a test function.
Then, by Lemma \ref{lem:210910-1} (3), we have
\begin{align*}
\langle f,\varphi\rangle
=
\lim_{l \to \infty}\langle e^{t_l\Delta}f,\varphi \rangle
&=
\lim_{l \to \infty}\sum_{Q\in{\mathcal D}({\mathbb R}^n)}\lambda^{t_l}_Q
\int_{{\mathbb R}^n}a^{t_l}_Q(x)\varphi(x)\,{\rm d}x
\end{align*}
from the definition of convergence in ${\mathcal S}'({\mathbb R}^n)$.
Once we fix $m$, we have
\[
\lambda^{t_l}_Q
\lesssim
2^{\frac{mn}p}\|{\mathcal M}f\|_{{\mathcal M}^p_{q,r}}
\quad \mbox{and} \quad
\left|\int_{{\mathbb R}^n}a^{t_l}_Q(x)\varphi(x)\,{\rm d}x\right|
\le
\int_{3Q}|\varphi(x)|\,{\rm d}x.
\]
Additionally, by the equation
\[
\sum_{Q\in{\mathcal D}_m({\mathbb R}^n)}
2^{\frac{mn}p}\|{\mathcal M}f\|_{{\mathcal M}^p_{q,r}}
\int_{3Q}|\varphi(x)|\,{\rm d}x
=
3^n2^{\frac{mn}p}\|{\mathcal M}f\|_{{\mathcal M}^p_{q,r}}
\|\varphi\|_{L^1}
<\infty,
\]
we can use Fubini's theorem to obtain
\[
\sum_{m\in{\mathbb Z}}\int_{{\mathbb R}^n}
\left(\sum_{Q\in{\mathcal D}_m({\mathbb R}^n)}\lambda^{t_l}_Q a^{t_l}_Q(x)\right)
\varphi(x)\,{\rm d}x
=
\sum_{m\in{\mathbb Z}}\sum_{Q\in{\mathcal D}_m({\mathbb R}^n)}\lambda^{t_l}_Q
\int_{{\mathbb R}^n}a^{t_l}_Q(x)\varphi(x)\,{\rm d}x.
\]
Hereinafter, we use also abbreviation
\begin{equation*}
a_{m,l}
:=
\sum_{Q\in{\mathcal D}_m({\mathbb R}^n)}
\lambda_Q^{t_l}\int_{{\mathbb R}^n}a_Q^{t_l}(x)\varphi(x)\,{\rm d}x,
\end{equation*}
and we fix $0<\varepsilon\ll1$.

When $m\in{\mathbb Z}\cap(-\infty,0]$, we see that
\begin{align*}
|a_{m,l}|
\lesssim
\sum_{Q\in{\mathcal D}_m({\mathbb R}^n)}
2^{\frac{mn}p}\|{\mathcal M}f\|_{{\mathcal M}^p_{q,r}}
\int_{3Q}|\varphi(x)|\,{\rm d}x
\lesssim_\varphi
2^{\frac{mn}p}\|{\mathcal M}f\|_{{\mathcal M}^p_{q,r}}
\end{align*}
by the previous argument, and therefore
\begin{equation}\label{eq:210930-1}
\sum_{m=-\infty}^0|a_{m,l}|
\lesssim
\|{\mathcal M}f\|_{{\mathcal M}^p_{q,r}}.
\end{equation}

In addition, taking $0<\varepsilon\ll1$ by $K+1>n(1/(q_0-\varepsilon)-1)>0$, namely,
\[
n+K+1>\frac n{q_0-\varepsilon},
\]
by Lemma \ref{lem:210924-1}, we obtain
\begin{equation}\label{eq:210930-2}
\sum_{m=1}^\infty|a_{m,l}|
\lesssim
\left\|\sum_{Q\in{\mathcal D}({\mathbb R}^n)}\lambda^t_Q\chi_{3Q}\right\|_{{\mathcal M}^p_{q,r}} 
\lesssim
\|{\mathcal M}f\|_{{\mathcal M}^p_{q,r}}.
\end{equation}
Thus, by (\ref{eq:210930-1}) and (\ref{eq:210930-2}), we obtain
\[
\sum_{m=-\infty}^\infty|a_{m,l}|
\lesssim
\|{\mathcal M}f\|_{{\mathcal M}^p_{q,r}}.
\]
As a consequence, applying the Lebesgue convergence theorem, we obtain
\[
\lim_{l \to \infty}\sum_{m=-\infty}^\infty a_{m,l}
=
\sum_{m=-\infty}^\infty\lim_{l \to \infty}a_{m,l}.
\]
Hence, using Fubini's theorem again, we have
\begin{align*}
\langle f,\varphi \rangle
&=
\sum_{m=-\infty}^\infty
\left(
\lim_{l\to\infty}\sum_{Q\in{\mathcal D}_m({\mathbb R}^n)}\lambda^{t_l}_Q
\int_{{\mathbb R}^n}a^{t_l}_Q(x)\varphi(x)\,{\rm d}x
\right)\\
&=
\sum_{m=-\infty}^\infty
\left(
\lim_{l\to\infty}
\int_{{\mathbb R}^n}
\left(\sum_{Q\in{\mathcal D}_m({\mathbb R}^n)}\lambda^{t_l}_Q a^{t_l}_Q(x)\right)\varphi(x)
\,{\rm d}x
\right)\\
&=
\sum_{m=-\infty}^\infty\sum_{Q \in {\mathcal D}_m({\mathbb R}^n)}\lim_{l \to \infty}
\left(
\int_{{\mathbb R}^n}\lambda^{t_l}_Q a^{t_l}_Q(x)\varphi(x)\,{\rm d}x
\right)\\
&=
\sum_{m=-\infty}^\infty\sum_{Q \in {\mathcal D}_m({\mathbb R}^n)}
\int_{{\mathbb R}^n}\lambda_Q a_Q(x)\varphi(x)\,{\rm d}x
=
\left<\sum_{Q \in {\mathcal D}({\mathbb R}^n)}\lambda_Q a_Q,\varphi \right>.
\end{align*}
Consequently, we obtain the desired result.
\end{proof}

\section{Application to the Olsen inequality}\label{s6}

As an application of Theorems \ref{thm:200308-1} and \ref{thm:200308-2}, we can reprove the following Olsen inequality about the fractional integral operator $I_\alpha$, $0<\alpha<n$, defined by
\begin{equation*}
I_\alpha f(x)
\equiv
\int_{{\mathbb R}^n}\frac{f(y)}{|x-y|^{n-\alpha}}\,{\rm d}y
\end{equation*}
for a measurable function $f$ defined on ${\mathbb R}^n$.

\begin{theorem}\label{thm:Olsen Morrey-Lorentz space}

Let $0<\alpha<n$, $1<p_1\le p_0<\infty$, $0<p_2\le\infty$, $1<q_1\le q_0<\infty$, $1<r_1\le r_0<\infty$ and $1<r_2\le\infty$.
Assume that
\begin{equation*}
r_1<q_1, \quad
\frac1{q_0}\le\frac\alpha n<\frac1{p_0}, \quad
\displaystyle\frac1{r_0}=\frac1{q_0}+\frac1{p_0}-\frac\alpha n.
\end{equation*}
If we suppose either {\rm (1)} or {\rm (2)} as follows{\rm :}
\begin{itemize}
\item[(1)] $0<r_2,p_2<\infty$ and $\dfrac{r_0}{p_0}=\dfrac{r_1}{p_1}=\dfrac{r_2}{p_2}$,
\item[(2)] $r_2=p_2=\infty$ and $\dfrac{r_0}{p_0}=\dfrac{r_1}{p_1}$,
\end{itemize}
then we have
\begin{equation*}
\|g\cdot I_\alpha f\|_{{\mathcal M}^{r_0}_{r_1,r_2}}
\lesssim
\|g\|_{{\rm W}{\mathcal M}^{q_0}_{q_1}} \|f\|_{{\mathcal M}^{p_0}_{p_1,p_2}},
\end{equation*}
for any measurable functions $f$ and $g$ in ${\mathcal M}^{p_0}_{p_1,p_2}({\mathbb R}^n)$ and ${\rm W}{\mathcal M}^{q_0}_{q_1}({\mathbb R}^n)$, respectively.
\end{theorem}

If we rewrite by $f\mapsto|\nabla u|$ in this theorem, the following norm estimate with a weight, which is extension of the Fefferman-Phong inequality \cite[p. 143]{Fefferman83}, can be obtained.

\begin{theorem}
Let $n\ge3$ and $1<q\le\dfrac n2<\infty$.
Then there exists a constant $K>0$ such that
\begin{equation*}
\int_{{\mathbb R}^n}|u(x)|^2V(x)\,{\rm d}x
\le
K\|V\|_{{\rm W}{\mathcal M}^{\frac n2}_q}
\int_{{\mathbb R}^n}|\nabla u(x)|^2\,{\rm d}x
\end{equation*}
for all measurable functions $u$ with $\nabla u\in(L^2({\mathbb R}^n))^n$ and non-negative measurable functions $V\in {\rm W}{\mathcal M}^{\frac n2}_q({\mathbb R}^n)$.
\end{theorem}

This theorem follows the necessary condition of the positivity of the Schr\"odinger operator $L\equiv-\Delta-V$ for the potential $V\ge0$; when $0<K\|V\|_{{\rm W}{\mathcal M}^{\frac n2}_q}\le1$, using the integration by parts, we see that
\begin{align*}
(Lu,u)_{L^2}
=
\int_{{\mathbb R}^n}|\nabla u(x)|^2\,{\rm d}x-\int_{{\mathbb R}^n}|u(x)|^2V(x)\,{\rm d}x
\ge
\left(1-K\|V\|_{{\rm W}{\mathcal M}^{\frac n2}_q}\right)
\int_{{\mathbb R}^n}|\nabla u(x)|^2\,{\rm d}x
\ge
0.
\end{align*}

On the other hand, considering the embedding ${\mathcal M}^{q_0}_{q_1}({\mathbb R}^n)\hookrightarrow{\rm W}{\mathcal M}^{q_0}_{q_1}({\mathbb R}^n)$ and the case $p_2=p_1$ and $q_2=q_1$ in Theorem \ref{thm:Olsen Morrey-Lorentz space}, we can obtain the original Olsen inequality given in \cite{Olsen95}.
In particular, in the case (2) in Theorem \ref{thm:Olsen Morrey-Lorentz space}, we can rewrite as follows:

\begin{theorem}

Let $0<\alpha<n$, $1<p_1\le p_0<\infty$, $1<q_1\le q_0<\infty$ and $1<r_1\le r_0<\infty$.
Assume that
\begin{equation*}
r_1<q_1, \quad
\frac1{q_0}\le\frac\alpha n<\frac1{p_0}, \quad
\displaystyle\frac1{r_0}=\frac1{q_0}+\frac1{p_0}-\frac\alpha n, \quad
\frac{r_0}{p_0}=\frac{r_1}{p_1}.
\end{equation*}
Then we have
\begin{equation*}
\|g\cdot I_\alpha f\|_{{\rm W}{\mathcal M}^{r_0}_{r_1}}
\lesssim
\|g\|_{{\rm W}{\mathcal M}^{q_0}_{q_1}} \|f\|_{{\rm W}{\mathcal M}^{p_0}_{p_1}},
\end{equation*}
for any measurable functions $f$ and $g$ in ${\rm W}{\mathcal M}^{p_0}_{p_1}({\mathbb R}^n)$ and ${\rm W}{\mathcal M}^{q_0}_{q_1}({\mathbb R}^n)$, respectively.
\end{theorem}

To prove Theorem \ref{thm:Olsen Morrey-Lorentz space}, the author employed the generalization for Adams theorem \cite{Adams75} and two lemmas mentioned in \cite{IST14} as follows.

\begin{proposition}{\cite[Proposition 3]{Hatano20-2}}\label{prop:Adams Morrey-Lorentz space}
Let $0<\alpha<n$, $1<q\le p<\infty$, $1<t\le s<\infty$ and $0<r,u\le\infty$.
Assume that $\dfrac1s=\dfrac1p-\dfrac\alpha n$.
If we suppose either {\rm (1)} or {\rm (2)} as follows:
\begin{itemize}
\item[(1)] $0<r,u<\infty$ and $\dfrac sp=\dfrac tq=\dfrac ur$,
\item[(2)] $r=u=\infty$ and $\dfrac sp=\dfrac tq$,
\end{itemize}
then we have
\begin{equation*}
\|I_\alpha f\|_{{\mathcal M}^s_{t,u}}
\lesssim
\|f\|_{{\mathcal M}^p_{q,r}},
\end{equation*}
for any measurable function $f$ in ${\mathcal M}^p_{q,r}({\mathbb R}^n)$.
\end{proposition}

It is well known that Hardy and Littlewood \cite{HaLi28} and Sobolev \cite{Sobolev38} proved the boundedness of a fractional integral operator on Lebesgue spaces, which is called the Hardy-Littlewood-Sobolev inequality: if $0<\alpha<n$ and $1<p<s<\infty$ satisfies $1/s=1/p-\alpha/n$, then
\begin{equation}\label{eq:HLS-ineq}
\|I_\alpha f\|_{L^s}
\lesssim
\|f\|_{L^p}
\end{equation}
for all $f\in L^p({\mathbb R}^n)$ (see, e.g., \cite[Theorem 1.2.3]{Grafakos14-2}).

\begin{lemma}[{\cite[Lemma 4.1]{IST14}}]\label{lem:120922-1}
There exists a constant depending only on $n$ and $\alpha$ such that, for every cube $Q$, we have $I_\alpha \chi_Q(x) \ge C\ell(Q)^\alpha\chi_Q(x)$ for all $x \in Q$.
\end{lemma}

\begin{lemma}[{\cite[Lemma 4.2]{IST14}}]\label{lem:200507-1}
Let $K=0,1,2,\ldots$.
Suppose that $A$ is an $L^\infty({\mathbb R}^n) \cap {\mathcal P}_K({\mathbb R}^n)^\perp$-function supported on a cube $Q$.
Then,
\begin{equation*}
|I_\alpha A(x)|
\le C_{\alpha,K}
\|A\|_{L^\infty}\ell(Q)^\alpha
\sum_{k=1}^\infty \frac{1}{2^{k(n+K+1-\alpha)}}\chi_{2^kQ}(x),
\quad x \in {\mathbb R}^n.
\end{equation*}
\end{lemma}

We remark that Lemma \ref{lem:200507-1} is proved by a method akin to Lemma \ref{lem:210803-11}.

Now we start the proof of Theorem \ref{thm:Olsen Morrey-Lorentz space}.
We may assume that  $f\ge0$ because the integral kernel of $I_\alpha$ is positive.
Now, we prove Theorem \ref{thm:Olsen Morrey-Lorentz space}.
Note that by Fatou's lemma and the Fatou property for Morrey-Lorentz spaces (see Lemma \ref{lem:Fatou Mpqr}),
\begin{equation*}
\|g\cdot I_\alpha f\|_{{\mathcal M}^{r_0}_{r_1,r_2}}
\le
\left\|g\cdot\liminf_{m\to\infty}I_\alpha f_m\right\|_{{\mathcal M}^{r_0}_{r_1,r_2}}
\le
\liminf_{m\to\infty}\|g\cdot I_\alpha f_m\|_{{\mathcal M}^{r_0}_{r_1,r_2}}
\end{equation*}
for all $f\in{\mathcal M}^{p_0}_{p_1,p_2}({\mathbb R}^n)$ and $g\in{\mathcal M}^{q_0}_{q_1,q_2}({\mathbb R}^n)$, where
\begin{equation*}
f_m:=f\chi_{B(m)}\chi_{[0,m)}(|f|) \in L^\infty_{\rm c}({\mathbb R}^n),
\quad m\in{\mathbb N}.
\end{equation*}
Then, we may assume that $f \in L^\infty_{\rm c}({\mathbb R}^n)$ is a positive measurable function in view of the positivity of the integral kernel.
We decompose $f$ according to Theorem \ref{thm:200308-2} with  sufficiently large $K\gg1$;
$f=\sum_{j=1}^\infty \lambda_j a_j$ converges in $L^w({\mathbb R}^n)$ for all $w\in(1,\infty)$, where
$\{Q_j\}_{j=1}^\infty \subset {\mathcal D}({\mathbb R}^n)$,
$\{a_j\}_{j=1}^\infty \subset L^\infty({\mathbb R}^n) \cap {\mathcal P}_K({\mathbb R}^n)^\perp$,
and $\{\lambda_j\}_{j=1}^\infty \subset[0,\infty)$ fulfill (\ref{eq:200308-2.1}).\\

Here, we claim that
\begin{equation}\label{eq:210827-2}
I_\alpha f(x)
=
\sum_{j=1}^\infty\lambda_jI_\alpha a_j(x),
\quad \text{a.e. $x\in{\mathbb R}^n$}.
\end{equation}
In fact, fixing $w\in(1,n/\alpha)$ and then choosing $w^\ast\in(1,\infty)$ satisfying $1/w^\ast=1/w-\alpha/n$, we have
\begin{align*}
\left\|I_\alpha f-\sum_{j=1}^N\lambda_jI_\alpha a_j\right\|_{L^{w^\ast}}
=
\left\|I_\alpha\left[f-\sum_{j=1}^N\lambda_ja_j\right]\right\|_{L^{w^\ast}}
\lesssim
\left\|f-\sum_{j=1}^N\lambda_ja_j\right\|_{L^w}
\to0
\end{align*}
as $N\to\infty$ from the Hardy-Littlewood-Sobolev inequality (see \eqref{eq:HLS-ineq}).
We finish the proof of \eqref{eq:210827-2}.\\

Then, by Lemma \ref{lem:200507-1}, we obtain
\[
|g(x)I_\alpha f(x)|
\lesssim
\sum_{j,k=1}^\infty
\frac{\lambda_j}{2^{k(n+K+1-\alpha)}}\ell(Q_j)^\alpha|g(x)|\chi_{2^kQ_j}(x).
\]
Therefore, we conclude that
\[
\|g \cdot I_\alpha f\|_{{\mathcal M}^{r_0}_{r_1,r_2}}
\lesssim
\|g\|_{{\rm W}{\mathcal M}^{q_0}_{q_1}}
\left\|
\sum_{j,k=1}^\infty
\frac{\lambda_j\ell(2^kQ_j)^{\alpha-\frac{n}{q_0}}}{2^{k(n+L+1)}}
\cdot
\frac{\ell(2^kQ_j)^{\frac{n}{q_0}}}{\|g\|_{{\rm W}{\mathcal M}^{q_0}_{q_1}}}
|g|\chi_{2^kQ_j}
\right\|_{{\mathcal M}^{r_0}_{r_1,r_2}}.
\]
For each $(j,k) \in {\mathbb N} \times {\mathbb N}$, write 
$$
\kappa_{jk}
:=
\frac{\lambda_j\ell(2^kQ_j)^{\alpha-\frac{n}{q_0}}}{2^{k(n+L+1)}},
\quad
b_{jk}
:=
\frac{\ell(2^kQ_j)^{\frac{n}{q_0}}}{\|g\|_{{\rm W}{\mathcal M}^{q_0}_{q_1}}}
|g|\chi_{2^kQ_j}.
$$
Then,
\[
\sum_{j,k=1}^\infty
\frac{\lambda_j\ell(2^kQ_j)^{\alpha-\frac{n}{q_0}}}{2^{k(n+K+1)}}
\cdot
\frac{\ell(2^kQ_j)^{\frac{n}{q_0}}}{\|g\|_{{\rm W}{\mathcal M}^{q_0}_{q_1}}}
|g|\chi_{2^kQ_j}
=
\sum_{j,k=1}^\infty\kappa_{jk}b_{jk},
\]
each $b_{jk}$ is supported on a cube $2^kQ_j$ and
\[
\|b_{jk}\|_{{\rm W}{\mathcal M}^{q_0}_{q_1}}
\le 
|2^kQ_j|^{\frac1{q_0}}.
\]

Observe that $\chi_{2^kQ_j} \le 2^{kn}M\chi_{Q_j}$.
Hence, if we choose $v,\theta\in{\mathbb R}$ such that
$$
K>\alpha-\frac{n}{q_0}-1+\theta n-n, \quad
\theta>\frac1v\ge\frac1{\min(r_1,r_2)}, \quad
0<v\le1,
$$
then we have
\begin{align*}
\left\|
\left(\sum_{j,k=1}^\infty(\kappa_{jk}\chi_{2^kQ_j})^v\right)^{\frac1v}
\right\|_{{\mathcal M}^{r_0}_{r_1,r_2}}
&=
\left\|\left(
\sum_{j,k=1}^\infty
\left(\frac{\lambda_j\ell(2^kQ_j)^{\alpha-\frac{n}{q_0}}}{2^{k(n+K+1)}}\chi_{2^kQ_j}\right)^v
\right)^{\frac1v}\right\|_{{\mathcal M}^{r_0}_{r_1,r_2}}\\
&\lesssim
\left\|\left(
\sum_{j=1}^\infty\left(\lambda_j\ell(Q_j)^{\alpha-\frac{n}{q_0}}(M\chi_{Q_j})^\theta\right)^v
\right)^{\frac1v}\right\|_{{\mathcal M}^{r_0}_{r_1,r_2}}\\
&=
\left\|\left\{
\sum_{j=1}^\infty  
\left(
M
\left[
\left(\lambda_j\ell(Q_j)^{\alpha-\frac n{q_0}}\chi_{Q_j}\right)^{\frac1\theta}
\right]
\right)^{\theta v}
\right\}^{\frac1v}\right\|_{{\mathcal M}^{r_0}_{r_1,r_2}}.
\end{align*}
By virtue of Proposition \ref{prop:200324-1}, the Fefferman-Stein inequality for Morrey-Lorentz spaces, alongside $f_j=(\lambda_j\ell(Q_j)^{\alpha-n/q_0}\chi_{Q_j})^{1/\theta}$, we can remove the maximal operator and obtain
\[
\left\|
\left(\sum_{j,k=1}^\infty(\kappa_{jk}\chi_{2^kQ_j})^v\right)^{\frac1v}
\right\|_{{\mathcal M}^{r_0}_{r_1,r_2}}
\lesssim
\left\|\left(
\sum_{j=1}^\infty\left(\lambda_j \ell(Q_j)^{\alpha-\frac{n}{q_0}}\chi_{Q_j}\right)^v
\right)^{\frac1v}\right\|_{{\mathcal M}^{r_0}_{r_1,r_2}}.
\]
We distinguish two cases here:
\begin{itemize}
\item[(1)] If $\alpha=n/q_0$, then $p_0=r_0$, $p_1=r_1$, and $p_2=r_2$.
Thus, we can use (\ref{eq:200308-2.1}).\\
\item[(2)] If $\alpha>n/q_0$, then by Proposition \ref{prop:Adams Morrey-Lorentz space} and Lemma \ref{lem:120922-1}, we obtain
\begin{align*}
&\left\|\left(
\sum_{j=1}^\infty\left(\lambda_j \ell(Q_j)^{\alpha-\frac{n}{q_0}}\chi_{Q_j}\right)^v
\right)^{\frac1v}\right\|_{{\mathcal M}^{r_0}_{r_1,r_2}}\\
&\lesssim
\left\|\left(
I_{\left(\alpha-\frac{n}{q_0}\right)v}
\left[\sum_{j=1}^\infty\left(\lambda_j \chi_{Q_j}\right)^v\right]
\right)^{\frac1v}\right\|_{{\mathcal M}^{r_0}_{r_1,r_2}}
\lesssim
\left\|
\left(\sum_{j=1}^\infty(\lambda_j\chi_{Q_j})^v\right)^{\frac1v}
\right\|_{{\mathcal M}^{p_0}_{p_1,p_2}}.
\end{align*}
Thus, we can still use (\ref{eq:200308-2.1}).
\end{itemize}
Consequently, we obtain
\begin{align}\label{eq:190820-5}
\left\|
\left(\sum_{j,k=1}^\infty(\kappa_{jk}\chi_{2^kQ_j})^v\right)^{\frac1v}
\right\|_{{\mathcal M}^{r_0}_{r_1,r_2}}
\lesssim
\left\|f \right\|_{{\mathcal M}^{p_0}_{p_1,p_2}}
<\infty.
\end{align}
Observe also that $q_0>r_0$ and $q_1>r_1$.
Thus, by Theorem \ref{thm:200308-1} and (\ref{eq:190820-5}), it follows that
\begin{align*}
\|g \cdot I_\alpha f\|_{{\mathcal M}^{r_0}_{r_1,r_2}}
&\lesssim
\|g\|_{{\rm W}{\mathcal M}^{q_0}_{q_1}}
\left\|
\left(\sum_{j,k=1}^\infty(\kappa_{jk}\chi_{2^kQ_j})^v\right)^{\frac1v}
\right\|_{{\mathcal M}^{r_0}_{r_1,r_2}}
\lesssim
\|g\|_{{\rm W}{\mathcal M}^{q_0}_{q_1}}
\left\|f \right\|_{{\mathcal M}^{p_0}_{p_1,p_2}}. 
\end{align*}

\section*{Acknowledgement}

The authors would like to thank Professor Yoshihiro Sawano for his careful reading of the manuscript. 
The author was financially supported by a Research Fellowship from the Japan Society for the Promotion of Science for Young Scientists (21J12129).

\end{document}